\documentclass[11pt]{amsart}

\usepackage[T1]{fontenc}
\usepackage[utf8]{inputenc}
\usepackage{amsmath,amssymb,amsthm,mathtools}
\usepackage{booktabs,tabularx}
\usepackage{newtxtext,newtxmath}
\usepackage[a4paper,margin=0.92in]{geometry}
\usepackage{microtype}
\usepackage{enumitem}
\usepackage{xcolor}
\usepackage{hyperref}

\definecolor{LinkBlue}{HTML}{173B57}
\definecolor{CiteRed}{HTML}{7A2E2E}
\hypersetup{
  colorlinks=true,
  linkcolor=LinkBlue,
  citecolor=CiteRed,
  urlcolor=LinkBlue,
  pdfdisplaydoctitle=true,
  pdflang={en-US},
  pdftitle={Heisenberg Uniqueness Pairs for a Hyperbola Branch: Supercritical Nonuniqueness for Shifted Lattice Crosses},
  pdfauthor={Zhiqiang Wan},
  pdfsubject={Supercritical nonuniqueness for shifted lattice crosses on a hyperbola branch},
  pdfkeywords={Heisenberg uniqueness pairs; shifted lattice cross; transfer operator; Klein-Gordon equation}
}

\newtheorem{theorem}{Theorem}[section]
\newtheorem{proposition}[theorem]{Proposition}
\newtheorem{lemma}[theorem]{Lemma}
\newtheorem{corollary}[theorem]{Corollary}
\theoremstyle{definition}
\newtheorem{definition}[theorem]{Definition}
\theoremstyle{remark}
\newtheorem{remark}[theorem]{Remark}

\newcommand{\R}{\mathbb R}
\newcommand{\C}{\mathbb C}
\newcommand{\Z}{\mathbb Z}
\newcommand{\N}{\mathbb N}
\newcommand{\T}{\mathbb T}
\newcommand{\one}{\mathbf 1}
\newcommand{\cA}{\mathcal A}
\newcommand{\cE}{\mathcal E}
\newcommand{\cL}{\mathcal L}
\newcommand{\cP}{\mathcal P}
\newcommand{\Var}{\operatorname{Var}}
\newcommand{\Ran}{\operatorname{Ran}}
\newcommand{\ess}{\mathrm{ess}}
\newcommand{\norm}[1]{\left\lVert#1\right\rVert}
\newcommand{\dd}{\,\mathrm{d}}

\setlist[enumerate]{leftmargin=1.65em,itemsep=.12em,topsep=.25em}
\setlist[itemize]{leftmargin=1.5em,itemsep=.12em,topsep=.25em}
\allowdisplaybreaks

\title[Supercritical nonuniqueness on a hyperbola branch]
{Heisenberg Uniqueness Pairs for a Hyperbola Branch:\\
Supercritical Nonuniqueness for Shifted Lattice Crosses}
\author{Zhiqiang Wan}
\address{School of Mathematical Sciences, University of Science and Technology of China,
No. 96 Jinzhai Road, Baohe District, Hefei, Anhui Province, China}
\email{ZhiQiang\_Wan576@mail.ustc.edu.cn}

\begin{document}
\frenchspacing

\begin{abstract}
We study Heisenberg uniqueness for the positive hyperbola branch and
shifted lattice crosses in the supercritical regime
$q=\alpha\gamma>1$.  We resolve the infinite-dimensionality clause of
the arbitrary-shift problem posed by Giri and Manna: for arbitrary
shifts on both arms, the normalized pre-annihilator is
infinite-dimensional.  More precisely, every
$v\in BV((1,q))$ has a global $BV$ pre-annihilating extension; the
extension is unique unless both twisting phases are trivial, in which
case its ambiguity is one-dimensional.  The proof reduces the
annihilation conditions to a graph equation for a twisted
Perron--Frobenius operator and combines a phase-uniform Lasota--Yorke
estimate with peripheral spectral rigidity.  We also give an exact
operator-theoretic normal form for the entire $L^1$ pre-annihilator in
terms of the maximal convergence domain of the associated Green series.
Writing $Q$ for the twisted product and $A$ for the forcing operator, we
show that $Q$ has the closed unit disk as its spectrum on $L^1((0,1))$,
that $\Ran(I-Q)$ is not closed, and that, outside a countable set of
algebraic values of $q>1$, the operator
$\sum_{j=0}^{N-1}Q^jA:L^1((1,q))\to L^1((0,1))$ has norm $2N$ for every
$N\ge1$.
\end{abstract}


\maketitle

\section{Introduction}

A Heisenberg uniqueness pair consists of a curve $\Gamma\subset\R^2$
and a set $\Lambda\subset\R^2$ with the following property: if a finite
complex Borel measure $\mu$, supported on $\Gamma$ and absolutely
continuous with respect to arc length, has Fourier transform vanishing on
$\Lambda$, then $\mu=0$.  Hedenmalm and Montes--Rodr\'{\i}guez introduced
this formulation in connection with the uncertainty principle and the
Klein--Gordon equation \cite{HM2011}.  The definition has several basic
stability properties.  Since $\widehat\mu$ is continuous, only the closure
of $\Lambda$ matters, and enlarging $\Lambda$ preserves uniqueness.  We
shall also use the standard affine invariances from
\cite[(inv-1)--(inv-2)]{HM2011}.  For $x_*,\xi_*\in\R^2$,
\begin{equation*}
 \big(\Gamma+\{x_*\},\Lambda+\{\xi_*\}\big)\ \mathrm{is\ a\ HUP}
 \quad\Longleftrightarrow\quad
 (\Gamma,\Lambda)\ \mathrm{is\ a\ HUP}.
\end{equation*}
If $T:\R^2\to\R^2$ is an invertible linear transformation with adjoint
$T^*$, then
\begin{equation*}
 \big(T^{-1}(\Gamma),T^*(\Lambda)\big)\ \mathrm{is\ a\ HUP}
 \quad\Longleftrightarrow\quad
 (\Gamma,\Lambda)\ \mathrm{is\ a\ HUP}.
\end{equation*}
The independent translations of the two arms of the crosses studied
below are not simultaneous translations of this form; they produce two
genuine phase parameters.

The algebraic structure of the supporting curve gives a second
interpretation.  If a polynomial $P$ vanishes on $\Gamma$, then, for the
Fourier transform normalized in \eqref{eq:Fourier-convention},
\begin{equation*}
 P\!\left((\pi i)^{-1}\partial_\xi,
          (\pi i)^{-1}\partial_\eta\right)\widehat\mu(\xi,\eta)=0.
\end{equation*}
For the hyperbola $x_1x_2=1$, this identity becomes
\begin{equation*}
 \left(\partial_\xi\partial_\eta+\pi^2\right)\widehat\mu=0,
\end{equation*}
the $(1+1)$-dimensional Klein--Gordon equation in characteristic
coordinates.  A HUP for the hyperbola is therefore a uniqueness set for
the Fourier transforms of the relevant Klein--Gordon data.  This also
explains why lattice crosses are natural: their arms lie in the two
characteristic directions.  The theory has subsequently been developed
for parallel lines, conic sections, algebraic curves and quadratic
hypersurfaces; see
\cite{BlasiBabot2013,Bagchi2018,GiriSrivastava2017,JamingKellay2018,
GrochenigJaming2020,Sjolin2013,GoncalvesRamos2022}.

For the full hyperbola and the unshifted lattice cross
\begin{equation*}
 (a\Z\times\{0\})\cup(\{0\}\times b\Z),
\end{equation*}
Hedenmalm and Montes--Rodr\'{\i}guez proved that uniqueness holds exactly
when $ab\le1$ \cite{HM2011}.  Passing to the positive branch
$\Gamma_+=\{(t,t^{-1}):t>0\}$ changes the critical product: uniqueness
holds for $ab<4$, the pre-annihilator is one-dimensional for $ab=4$, and
it is infinite-dimensional for $ab>4$ \cite{CHM2014}.  The proof converts
the Fourier conditions into a
problem for Perron--Frobenius operators induced by Gauss-type maps.  The
connections with Hilbert transforms, one-sided Fourier information and
Gauss dynamics were developed further in \cite{HM2020,HM2021}.  Thus the
sharp density transition is encoded by the spectral behavior of a
reciprocal dynamical system.

A one-arm shift of the lattice cross was proposed in
\cite[Section~7(a)]{HM2011}.  Rational and more general perturbations
for the full hyperbola were considered in
\cite{GR2021,GRCorr2022,GM2024,RadchenkoRamos2024}.  For the positive
branch, the closest shifted predecessors are Giri's rational one-arm
construction \cite{Giri2020}, the arbitrary-shift subcritical theorem
of Giri and Manna \cite{GM2024}, and their later arithmetic results
\cite{GM2026}.  The last paper asks whether the pre-annihilator is
infinite-dimensional for arbitrary shifts throughout the supercritical
regime and how the entire pre-annihilator may be characterized.
Theorem~\ref{thm:intro-main} settles its infinite-dimensionality clause
for arbitrary shifts on both arms and every $q>1$.  The problem also
belongs to the broader theory of Fourier uniqueness and interpolation
represented by
\cite{BakanEtAl2021,HMHyperbolic2026,RV2019,RS2022,KNS2025}.

We next explain the normalization used here.  For $\alpha,\gamma>0$ and
real shifts $\theta_1,\theta_2$, set
\begin{equation*}
 \Lambda_{\alpha,\gamma,\theta_1,\theta_2}
 =
 \big((2\alpha\Z+\theta_1)\times\{0\}\big)
 \cup
 \big(\{0\}\times(2\gamma\Z+\theta_2)\big).
\end{equation*}
The physical spacings are $2\alpha$ and $2\gamma$, so their product is
$4\alpha\gamma$.  The critical branch product $4$ consequently becomes
$q=\alpha\gamma=1$.  At the level of the density, the substitution
$x=\alpha t$ gives the isometry
\begin{equation*}
 f\longmapsto F,\qquad F(x)=\alpha^{-1}f(x/\alpha).
\end{equation*}
It normalizes the first spacing to $2$, leaves
\begin{equation*}
 q=\alpha\gamma
\end{equation*}
as the only geometric parameter, and records the two shifts through
\begin{equation*}
 \zeta=e^{\pi i\theta_1/\alpha},\qquad
 \eta=e^{\pi i\theta_2/\gamma}.
\end{equation*}
We denote the isometric image of the original pre-annihilator by
$\cA_{q,\zeta,\eta}$ and call it the \emph{normalized
pre-annihilator}.  Thus the adjective ``normalized'' refers only to an
equivalent change of variables; it imposes no additional condition and
preserves both the dimension of the pre-annihilator and the HUP
alternative.  This is the shifted version of the scaling reductions in
\cite{HM2011,CHM2014}.

\begin{table}[t]
\caption{Closest results, in the normalized notation
$q=\alpha\gamma$.}
\label{tab:prior-results}
\small
\setlength{\tabcolsep}{3.2pt}
\renewcommand{\arraystretch}{1.1}
\begin{tabularx}{\textwidth}
{@{}>{\raggedright\arraybackslash}p{2.15cm}
>{\raggedright\arraybackslash}p{2.15cm}
>{\raggedright\arraybackslash}p{2.0cm}
>{\raggedright\arraybackslash}X@{}}
\toprule
Reference & shifts & parameter range & conclusion \\
\midrule
\cite{CHM2014}
 & none
 & $q>1$
 & infinite-dimensional pre-annihilator on $\Gamma_+$ \\
\cite{Giri2020}
 & one rational arm; $\alpha=1$, $\theta_1=2/p$, $\theta_2=0$,
 $p$ a positive integer
 & $q=\gamma>p$
 & infinite-dimensional pre-annihilator on $\Gamma_+$ \\
\cite{GM2024}
 & arbitrary two-arm
 & $q<1$
 & uniqueness, including the positive-branch formulation \\
\cite{GM2026}
 & arbitrary two-arm shifts
 & $q=1$
 & conditional modulus rigidity; the critical pre-annihilator has
 dimension at most one, but Theorem~1.2 does not establish a nonzero
 element \\
\cite{GM2026}
 & special rational shifts
 & selected ranges of $q>1$
 & special supercritical normal forms and infinite-dimensionality;
 leaves Open Problem~1.4 on $\Gamma_+$ \\
Present paper
 & arbitrary two-arm
 & every $q>1$
 & $BV$ extension theorem, infinite-dimensionality, and an
 operator-theoretic $L^1$ normal form on $\Gamma_+$ \\
\bottomrule
\end{tabularx}
\end{table}

The reduction to Gauss dynamics also singles out the regularity class
used in the proof.  Although $L^1(\R_+)$ is the ambient space for
densities of admissible measures, the positive-branch argument of
Canto--Mart\'{\i}n, Hedenmalm and Montes--Rodr\'{\i}guez shows that the
transfer operator acquires its useful spectral structure on $BV$: it
preserves bounded variation, its unimodular eigenspaces lie in $BV$, and
bounded-variation data on the middle interval admit global
pre-annihilating extensions \cite{CHM2014}.  The same regularity class is
retained in the later constructions \cite{GM2024,GM2026}.  Indeed, after
one-sided periodization and reciprocal inversion, the normalized
annihilation conditions are governed on $(0,1)$ by the countable-branch
map
\begin{equation*}
 T_q(x)=\{q/x\}_1.
\end{equation*}
Its inverse branches are monotone M\"obius maps.  Variation controls their
distortion and the jumps at branch endpoints, while the compact embedding
$BV((0,1))\hookrightarrow L^1((0,1))$ provides the strong--weak
compactness needed to separate the peripheral dynamics.

To obtain this spectral separation in the presence of shifts, we
establish a phase-uniform variation estimate.  The contraction of
variation originates in the work of Lasota and Yorke on piecewise
monotone maps \cite{LasotaYorke1973}.  Since $T_q$ has countably many
monotonicity intervals, the branch summation is organized in the
bounded-variation framework of Rychlik \cite{Rychlik1983}.  Once the
resulting strong--weak inequality is available, Hennion's compactness
criterion converts it into a bound for the essential spectral radius
\cite{Hennion1993}.  The unimodular branch factors $\zeta^n$ and
$\eta^n$ change phases but not absolute variations; Section~3 makes this
uniformity precise.

For $q>1$, the intervals $(0,1)$, $(1,q)$ and $(q,\infty)$ give the
three-piece decomposition
\begin{equation*}
 F=F_0+v+F_\infty.
\end{equation*}
We call $v=F|_{(1,q)}$ the \emph{middle restriction}, or the
\emph{middle datum}.  Thus $F$ has \emph{zero middle restriction}
precisely when it vanishes almost everywhere on $(1,q)$; no vanishing on
the other two intervals is implied.  The transfer-operator argument first
identifies the regular part of the normalized pre-annihilator.  We
therefore set
\begin{equation*}
 \cA^{BV}_{q,\zeta,\eta}
 :=\cA_{q,\zeta,\eta}\cap BV(\R_+)
\end{equation*}
and call this intermediate space the \emph{global $BV$ core}.
Equivalently, by Lemma~\ref{lem:tail-BV}, its elements are precisely the
normalized pre-annihilators whose restrictions to $(0,1)$ and $(1,q)$
belong to $BV((0,1))$ and $BV((1,q))$, respectively.  We first determine
the restriction map on this core; the Green-domain construction in
Section~5 then connects that result to the direct-sum normal form of the
entire $L^1$ pre-annihilator.

\begin{theorem}\label{thm:intro-main}
Let $\alpha,\gamma>0$, let
$\theta_1,\theta_2\in\R$, and suppose that
$q=\alpha\gamma>1$.
\begin{enumerate}[label=\textup{(\roman*)}]
 \item Every $v\in BV((1,q))$ is the restriction to $(1,q)$ of an
 element of $\cA^{BV}_{q,\zeta,\eta}$.
 \item If $(\zeta,\eta)\ne(1,1)$, this extension is unique within the
 $BV$ core.  If $(\zeta,\eta)=(1,1)$, the extensions form an affine
 line whose direction is a fixed nonzero annihilator with zero middle
 restriction.
 \item The pre-annihilator of the shifted lattice cross is
 infinite-dimensional.  Consequently,
 $(\Gamma_+,\Lambda_{\alpha,\gamma,\theta_1,\theta_2})$ is not a
 Heisenberg uniqueness pair.
\end{enumerate}
\end{theorem}

\begin{remark}
 Part~\textup{(iii)} of Theorem~\ref{thm:intro-main} settles the
 infinite-dimensionality clause of \cite[Open Problem~1.4]{GM2026} for
 arbitrary shifts.  Corollary~\ref{cor:L1-normal} gives an exact
 operator-theoretic parametrization of the entire $L^1$ pre-annihilator
 in terms of the maximal Green domain.  An intrinsic function-space
 description of
 $\mathfrak D_{q,\zeta,\eta}$ in a standard function scale remains
 open.  The graph-core equality in Remark~\ref{rem:graph-core} is a
 separate, strictly stronger density question, and the critical case
 $q=1$ lies outside the present supercritical regime.
\end{remark}

The theorem says that the restriction map from
$\cA^{BV}_{q,\zeta,\eta}$ to $BV((1,q))$ is surjective.  Its kernel is
trivial when at least one phase is nontrivial and one-dimensional in the
untwisted case.  Since $BV((1,q))$ is infinite-dimensional, the
nonuniqueness assertion follows immediately.  In particular, the shifts
affect the ambiguity of the extension but not its existence.
The primary HUP advance is the removal of both the arithmetic restriction
on the shifts and the restricted supercritical parameter range.  The
$BV$ extension and uniqueness statements are structural strengthenings;
the twisted peripheral rigidity, $L^1$ spectral description, and Green
growth estimates are the operator results that support the normal form.

A normal form for the entire $L^1$ pre-annihilator requires a further
step.  The periodization
identities reduce the problem to the graph equation
\begin{equation*}
 (I-Q_{q,\zeta,\eta})F_0=A_{q,\zeta,\eta}v.
\end{equation*}
The range of $I-Q_{q,\zeta,\eta}$ is not closed on $L^1((0,1))$, so the
$BV$ inverse does not extend boundedly to every $L^1$ datum.  For the
transfer and forcing operators defined in Section~2, consider the Green
partial sums
\begin{equation*}
 \mathsf G_Nv=\sum_{j=0}^{N-1}
 Q_{q,\zeta,\eta}^{\,j}A_{q,\zeta,\eta}v.
\end{equation*}
These sums need not converge for every $v\in L^1((1,q))$.  We call the
largest set of middle data for which they converge in $L^1((0,1))$ the
\emph{maximal Green domain}; see
Subsection~\ref{subsec:green-domain} for the precise definition.  On this
domain the limit solves the graph equation, and
$\ker(I-Q_{q,\zeta,\eta})$ gives exactly the possible correction.  This
yields an exact operator-theoretic $L^1$ normal form and identifies the
obstruction outside the $BV$ core.  It does not, by itself, describe the
maximal Green domain intrinsically in a standard function space.

The paper is organized as follows.
\begin{enumerate}
 \item Section~2 introduces the pre-annihilator, twisted periodizations,
 and the transfer-operator graph equation.
 \item Section~3 proves the twisted Lasota--Yorke inequality and the
 required quasi-compactness on bounded variation.
 \item Section~4 classifies the unit-circle point spectrum and proves
 supercritical infinite-dimensionality.
 \item Section~5 defines the maximal Green domain and constructs the
 normal form for the entire $L^1$ pre-annihilator, including the
 nonclosed-range and Green-growth
 statements.
\end{enumerate}

\section{Pre-annihilators and transfer operators}

Here and below, $\R_+=(0,\infty)$ and
$\N=\{1,2,\ldots\}$; identities between $L^1$ functions hold almost
everywhere, and a function assigned to one of the intervals in our
decomposition is extended by zero outside that interval.

\subsection{The branch model and duality}

We use the Fourier convention
\begin{equation}\label{eq:Fourier-convention}
 \widehat\mu(\xi,\eta)
 =\int_{\R^2}e^{\pi i(x\xi+y\eta)}\,\dd\mu(x,y).
\end{equation}
Let
\begin{equation*}
 \Gamma_+=\{(t,t^{-1}):t>0\}.
\end{equation*}
Every finite complex measure on $\Gamma_+$, absolutely continuous with
respect to arc length, may be written uniquely, up to null sets, in the
form
\begin{equation*}
 \int_{\Gamma_+}\Phi\,\dd\mu
 =\int_{\R_+} \Phi(t,t^{-1})f(t)\,\dd t
\end{equation*}
with $f\in L^1(\R_+)$.  The arc-length factor has simply been absorbed
into $f$.  More explicitly,
\begin{equation*}
 \dd s=\sqrt{1+t^{-4}}\,\dd t,
 \qquad
 \frac{\dd\mu}{\dd s}(t,t^{-1})
 =\frac{f(t)}{\sqrt{1+t^{-4}}}.
\end{equation*}
Thus this parametrization is an isometric linear correspondence between
$L^1(\R_+,\dd t)$ and the arc-length absolutely continuous finite
complex measures on $\Gamma_+$, equipped with the total-variation norm.

The vanishing of $\widehat\mu$ on the shifted cross introduced above is
equivalent to
\begin{align}
 \int_{\R_+}
 f(t)e^{\pi i(2m\alpha+\theta_1)t}\,\dd t&=0,
 &&m\in\Z,\label{eq:first-family}\\
 \int_{\R_+}
 f(t)e^{\pi i(2n\gamma+\theta_2)/t}\,\dd t&=0,
 &&n\in\Z.\label{eq:second-family}
\end{align}

\begin{definition}[Pre-annihilator]\label{def:preann}
The pre-annihilator associated with the shifted cross is
\begin{align*}
 \cA_{\alpha,\gamma,\theta_1,\theta_2}
 :=\big\{f\in L^1(\R_+):\;&\eqref{eq:first-family}
 \ \mathrm{and}\ \eqref{eq:second-family}\ \mathrm{hold}\big\}.
\end{align*}
\end{definition}

If $\cE_{\alpha,\gamma,\theta_1,\theta_2}$ denotes the linear span in
$L^\infty(\R_+)$ of the two exponential families in
\eqref{eq:first-family}--\eqref{eq:second-family}, then
\begin{equation*}
 \cA_{\alpha,\gamma,\theta_1,\theta_2}
 =\cE_{\alpha,\gamma,\theta_1,\theta_2}^{\perp}.
\end{equation*}
Consequently, the pair
$(\Gamma_+,\Lambda_{\alpha,\gamma,\theta_1,\theta_2})$ is a HUP if and
only if the pre-annihilator is $\{0\}$, or equivalently if
$\cE_{\alpha,\gamma,\theta_1,\theta_2}$ is weak-star dense in
$L^\infty(\R_+)$.  If the pre-annihilator has finite dimension $d$,
then that weak-star closure has codimension $d$.

\subsection{Koopman and Perron--Frobenius operators}

Let $T:I\to I$, $I=(0,1)$, be a nonsingular piecewise monotone map.
We use the preadjoint convention of
\cite[Subsection~3.2]{CHM2014}; the positive Gauss-type operator below
is the one used in \cite[Subsection~4.1]{CHM2014}.
The Koopman operator acts on observables by
\begin{equation*}
 U_T\varphi=\varphi\circ T.
\end{equation*}
Its preadjoint on $L^1(I)$ is the Perron--Frobenius operator
$\mathcal P_T$, defined by
\begin{equation*}
 \int_I(\mathcal P_T h)\varphi\,\dd x
 =\int_Ih(\varphi\circ T)\,\dd x.
\end{equation*}
If $\phi_a:J_a\to I_a$ are the inverse branches, then
\begin{equation*}
 (\mathcal P_Th)(x)
 =\sum_a\one_{J_a}(x)|\phi_a'(x)|h(\phi_a(x)).
\end{equation*}

The map relevant here is defined almost everywhere on $I=(0,1)$ by
\begin{equation*}
 T_q(y)=\{q/y\}_1,\qquad q>1,
\end{equation*}
where $\{\cdot\}_1$ is the fractional part in $[0,1)$.  At the
countable points $q/n\in I$, where this formula takes the value
$0\notin I$, we assign any value in $I$.  This does not change the
associated operators.  Modulo endpoints, the inverse branches are
\begin{equation*}
 \phi_n(x)=\frac q{n+x}.
\end{equation*}
The positive Perron--Frobenius operator is
\begin{equation*}
 (\cL_{1,q}h)(x)
 =\sum_{n\ge1}\frac q{(n+x)^2}
 h\!\left(\frac q{n+x}\right),
\end{equation*}
where $h$ is extended by zero outside $I$.  A shift in the lattice
inserts a branch phase:
\begin{equation*}
 (\cL_{\omega,q}h)(x)
 =\sum_{n\ge1}\omega^n\frac q{(n+x)^2}
 h\!\left(\frac q{n+x}\right),
 \qquad \omega\in\T.
\end{equation*}
This operator is generally not positivity-preserving, but
\begin{equation*}
 |\cL_{\omega,q}h|\le \cL_{1,q}|h|,
 \qquad
 \norm{\cL_{\omega,q}h}_1\le\norm h_1.
\end{equation*}

\subsection{The \texorpdfstring{$BV$}{BV} space}

For a complex-valued function on an interval $J$, let
\begin{equation*}
 \norm h_{BV(J)}=\norm h_{L^1(J)}+\Var_J(h).
\end{equation*}
We use representatives with one-sided traces.  Two elementary estimates
will be used repeatedly:
\begin{align}
 \norm h_{\infty,J}
 &\le |J|^{-1}\norm h_{L^1(J)}+\Var_J(h),\label{eq:BV-sup}\\
 \Var_I(\one_Jh)
 &\le \Var_J(h)+2\norm h_{\infty,J},
 \qquad J\subset I.\notag
\end{align}
The embedding $BV(I)\hookrightarrow L^1(I)$ is compact.
This is the standard setting for piecewise expanding transfer operators;
see, for example, \cite{Rychlik1983}.

We use the Calkin essential spectral radius
\begin{equation*}
 r_{\ess}(L;\mathsf B)
 =
 r_{\mathcal L(\mathsf B)/\mathcal K(\mathsf B)}
 \big(L+\mathcal K(\mathsf B)\big),
\end{equation*}
where $\mathcal K(\mathsf B)$ is the ideal of compact operators on
$\mathsf B$.  The following is the specialization of Hennion's
criterion used below.

\begin{theorem}[Hennion]\label{thm:Hennion}
Let $\mathsf B$ and $\mathsf B_w$ be complex Banach spaces such that
the inclusion
\begin{equation*}
 \mathsf B\hookrightarrow\mathsf B_w
\end{equation*}
is continuous and compact.  Suppose that $L$ is bounded on both spaces,
the two actions are compatible with the inclusion, and, for some
$0\le a<1$ and $b<\infty$,
\begin{align}
 \norm{Lu}_{\mathsf B}
 &\le a\norm u_{\mathsf B}+b\norm u_{\mathsf B_w},
 \label{eq:abstract-DF}\\
 \norm{Lu}_{\mathsf B_w}
 &\le\norm u_{\mathsf B_w}.
 \notag
\end{align}
Then
\begin{equation*}
 r_{\ess}(L;\mathsf B)\le a.
\end{equation*}
\end{theorem}

\begin{proof}
Restrict the weak norm to $\mathsf B$ by putting
\begin{equation*}
 |u|=\norm u_{\mathsf B_w}.
\end{equation*}
In Hennion's terminology, $L$ has property $\mathrm{DF}(r)$ if it is
compact from the strong norm to the weak norm and, for every $n\ge1$,
there are $R_n,r_n\ge0$ such that
\begin{equation*}
 \norm{L^nu}_{\mathsf B}
 \le R_n\norm u_{\mathsf B_w}+r_n\norm u_{\mathsf B},
 \qquad
 \liminf_{n\to\infty}r_n^{1/n}=r<r(L;\mathsf B);
\end{equation*}
see \cite[p.~627, definition of $\mathrm{DF}(r)$]{Hennion1993}.
The compactness condition holds because $L:\mathsf B\to\mathsf B$ is
bounded and $\mathsf B\hookrightarrow\mathsf B_w$ is compact.

Iteration of \eqref{eq:abstract-DF}, together with the weak contraction,
gives
\begin{equation*}
 \norm{L^nu}_{\mathsf B}
 \le a^n\norm u_{\mathsf B}
 +b\frac{1-a^n}{1-a}\norm u_{\mathsf B_w}.
\end{equation*}
Thus, when $a>0$, one may take
\begin{equation*}
 r_n=a^n,\qquad R_n=b\frac{1-a^n}{1-a}.
\end{equation*}
If $a<r(L;\mathsf B)$, then
\cite[p.~628, Corollaire~1]{Hennion1993} yields
\begin{equation*}
 r_{\ess}(L;\mathsf B)\le a.
\end{equation*}
If $a\ge r(L;\mathsf B)$, the same conclusion follows from
$r_{\ess}(L;\mathsf B)\le r(L;\mathsf B)$.  Finally, if $a=0$ and
$r(L;\mathsf B)>0$, apply the preceding argument with any
$0<\varepsilon<r(L;\mathsf B)$ in place of $a$ and let
$\varepsilon\downarrow0$; if $r(L;\mathsf B)=0$, the conclusion is
immediate.
\end{proof}

\subsection{Twisted periodization}

We retain the parameters $q,\zeta,\eta$ defined above.  With
$x=\alpha t$, define
\begin{equation*}
 F(x)=\alpha^{-1}f(x/\alpha).
\end{equation*}
This scaling preserves the $L^1$ norm.

\begin{lemma}\label{lem:periodization}
A function $f\in L^1(\R_+)$ belongs to
$\cA_{\alpha,\gamma,\theta_1,\theta_2}$ if and only if $F$ satisfies,
for almost every $x\in(0,1)$,
\begin{align*}
 \sum_{k=0}^\infty\zeta^kF(x+k)&=0,\\
 \sum_{j=0}^\infty\eta^j\frac q{(x+j)^2}
 F\!\left(\frac q{x+j}\right)&=0.
\end{align*}
Both series converge absolutely almost everywhere and define
$L^1((0,1))$ functions.
\end{lemma}

\begin{proof}
Make first the change of variables $x=\alpha t$.  For $m\in\Z$,
\begin{align*}
 \int_{\R_+}f(t)e^{\pi i(2m\alpha+\theta_1)t}\,\dd t
 &=\int_{\R_+}e^{2\pi imx}G_1(x)\,\dd x,\\
 G_1(x)&:=e^{\pi i\theta_1x/\alpha}F(x).
\end{align*}
The function $G_1$ belongs to $L^1(\R_+)$.  By the corrected
one-sided periodization lemma \cite[Lemma~2.1]{Corr2026}, the first
family of Fourier integrals vanishes for every $m\in\Z$ if and only if
\begin{equation*}
 \sum_{k=0}^{\infty}G_1(x+k)=0
 \qquad\mathrm{for\ a.e.}\ x\in(0,1).
\end{equation*}
Since $\zeta=e^{\pi i\theta_1/\alpha}$,
\begin{equation*}
 \sum_{k=0}^{\infty}G_1(x+k)
 =e^{\pi i\theta_1x/\alpha}
  \sum_{k=0}^{\infty}\zeta^kF(x+k).
\end{equation*}
The common factor is nonzero, so this is equivalent to the first
periodization in the statement.

For the second family, the same scaling followed by $y=q/x$ gives,
using $q=\alpha\gamma$,
\begin{align*}
 \int_{\R_+}f(t)e^{\pi i(2n\gamma+\theta_2)/t}\,\dd t
 &=\int_{\R_+}F(x)
   e^{\pi i(2nq+\alpha\theta_2)/x}\,\dd x\\
 &=\int_{\R_+}e^{2\pi iny}G_2(y)\,\dd y,\\
 G_2(y)&:=e^{\pi i\theta_2y/\gamma}
          \frac q{y^2}F\!\left(\frac qy\right).
\end{align*}
The inversion $y\mapsto q/y$ shows that $G_2\in L^1(\R_+)$.  Applying
\cite[Lemma~2.1]{Corr2026} to $G_2$, the second family vanishes for
every $n\in\Z$ if and only if
\begin{equation*}
 \sum_{j=0}^{\infty}G_2(x+j)=0
 \qquad\mathrm{for\ a.e.}\ x\in(0,1).
\end{equation*}
Since $\eta=e^{\pi i\theta_2/\gamma}$,
\begin{equation*}
 \sum_{j=0}^{\infty}G_2(x+j)
 =e^{\pi i\theta_2x/\gamma}
  \sum_{j=0}^{\infty}\eta^j\frac q{(x+j)^2}
  F\!\left(\frac q{x+j}\right),
\end{equation*}
which proves the second equivalence.

It remains to justify the asserted convergence.  Tonelli's theorem and
the same changes of variables give
\begin{align*}
 \int_0^1\sum_{k=0}^{\infty}|G_1(x+k)|\,\dd x
 &=\int_{\R_+}|F(u)|\,\dd u,\\
 \int_0^1\sum_{j=0}^{\infty}|G_2(x+j)|\,\dd x
 &=\int_{\R_+}\frac q{y^2}
   \left|F\!\left(\frac qy\right)\right|\,\dd y
  =\int_{\R_+}|F(u)|\,\dd u.
\end{align*}
Both right-hand sides are finite.  Hence both series converge absolutely
for almost every $x\in(0,1)$, and the sums of their absolute values
belong to $L^1((0,1))$.  The two applications of the corrected lemma
are equivalences, so the converse implication follows as well.
\end{proof}

This is the character-weighted form of the one-sided periodization
mechanism in \cite[Section~5]{HM2011}; the restriction to the fundamental
interval is the correction made in \cite{Corr2026}.

Define the weighted inversion
\begin{equation*}
 (J_qh)(x)=\frac q{x^2}h\!\left(\frac qx\right).
\end{equation*}
It is an isometric involution on $L^1(\R_+)$:
\begin{equation*}
 J_q^2=I,\qquad \norm{J_qh}_1=\norm h_1.
\end{equation*}
For $\omega\in\T$, define the one-sided periodization
\begin{equation*}
 (S_\omega h)(x)=\sum_{n\ge1}\omega^nh(x+n),
 \qquad 0<x<1.
\end{equation*}
If $h\in L^1((1,\infty))$, then
\begin{equation}\label{eq:S-contraction}
 \norm{S_\omega h}_{L^1((0,1))}
 \le\sum_{n\ge1}\int_0^1|h(x+n)|\,\dd x
 \le\norm h_{L^1((1,\infty))}.
\end{equation}
The two periodizations may now be written as
\begin{equation}\label{eq:short-periods}
 F|_{(0,1)}=-S_\zeta(F|_{(1,\infty)}),\qquad
 (J_qF)|_{(0,1)}
 =-S_\eta((J_qF)|_{(1,\infty)}).
\end{equation}
The restriction to the fundamental interval is essential; it is consistent
with the correction in \cite{Corr2026}.

\subsection{The \texorpdfstring{$L^1$}{L1} graph equation}

Assume $q>1$.  Following the three-piece decomposition used in
\cite[Proposition~8.5]{CHM2014}, write
\begin{equation*}
 F=F_0+v+F_\infty,
\end{equation*}
where $F_0$, $v$ and $F_\infty$ vanish almost everywhere outside
$(0,1)$, $(1,q)$ and $(q,\infty)$, respectively.
The second identity in \eqref{eq:short-periods} determines the tail:
\begin{equation}\label{eq:B-def}
 F_\infty=B_{\eta,q}(F_0+v),\qquad
 B_{\eta,q}h=-J_q\big[S_\eta(J_qh)\big].
\end{equation}
The expression in brackets is initially defined on $(0,1)$ and then
extended by zero.  Since $h$ vanishes almost everywhere outside
$(0,q)$, the function $J_qh$ vanishes almost everywhere outside
$(1,\infty)$, and $B_{\eta,q}h$ vanishes almost everywhere outside
$(q,\infty)$.
Moreover, \eqref{eq:S-contraction} and the isometry of $J_q$ give
\begin{equation}\label{eq:B-bound}
 \norm{B_{\eta,q}h}_1\le\norm h_1.
\end{equation}
For $t>q$, the formula is
\begin{equation}\label{eq:B-explicit}
 (B_{\eta,q}h)(t)
 =-\sum_{j\ge1}\eta^j\frac{q^2}{(q+jt)^2}
 h\!\left(\frac{qt}{q+jt}\right).
\end{equation}
This is the character-weighted version of the positive-branch tail
operator implicit in
\cite[Proposition~8.5]{CHM2014}.  Since the
global $BV$ conclusion is not supplied there, we record the short estimate
needed here.

\begin{lemma}\label{lem:tail-BV}
Suppose that $h\in L^1((0,q))$, with
$h|_{(0,1)}\in BV((0,1))$ and $h|_{(1,q)}\in BV((1,q))$.  Then
\begin{equation*}
 B_{\eta,q}h\in BV((q,\infty)),
\end{equation*}
and the resulting map from
$BV((0,1))\oplus BV((1,q))$ to $BV((q,\infty))$ is bounded, uniformly in
$\eta\in\T$.  Consequently, after extending $h$ by zero outside
$(0,q)$, one has $h+B_{\eta,q}h\in BV(\R_+)$.
\end{lemma}

\begin{proof}
Write
\begin{equation*}
 h_-=h|_{(0,1)},\qquad h_+=h|_{(1,q)},
\end{equation*}
and choose their one-sided-trace representatives.  Their piecewise union
$\widetilde h$ satisfies
\begin{align*}
 \Var_{(0,q)}(\widetilde h)
 &\le
 \Var_{(0,1)}(h_-)+\Var_{(1,q)}(h_+)
 +|h_-(1-)-h_+(1+)|\\
 &\le
 \Var_{(0,1)}(h_-)+\Var_{(1,q)}(h_+)
 +\norm{h_-}_{\infty,(0,1)}
 +\norm{h_+}_{\infty,(1,q)}.
\end{align*}
By \eqref{eq:BV-sup}, both this quantity and
$\norm{\widetilde h}_{\infty,(0,q)}$ are bounded by a constant depending
only on $q$ times
\begin{equation*}
 \norm{h_-}_{BV((0,1))}+\norm{h_+}_{BV((1,q))}.
\end{equation*}
For $j\ge1$, put
\begin{equation*}
 s_j(t)=\frac{qt}{q+jt},\qquad
 w_j(t)=\frac{q^2}{(q+jt)^2},\qquad t>q.
\end{equation*}
Then $s_j$ maps $(q,\infty)$ increasingly onto
\begin{equation*}
 K_j=\left(\frac q{j+1},\frac qj\right),
\end{equation*}
the intervals $K_j$ partition $(0,q)$, and $s_j'=w_j$.  Moreover,
\begin{equation*}
 \norm{w_j}_{\infty,(q,\infty)}
 =\Var_{(q,\infty)}(w_j)=\frac1{(j+1)^2}.
\end{equation*}
The change of variables $s=s_j(t)$ and the $BV$ product and
monotone-composition rules give
\begin{align*}
 \sum_{j\ge1}\norm{w_j(\widetilde h\circ s_j)}_1
 &=\norm{\widetilde h}_1,\\
 \sum_{j\ge1}
 \Var_{(q,\infty)}\big(w_j(\widetilde h\circ s_j)\big)
 &\le \sum_{j\ge1}\Var_{K_j}(\widetilde h)
 +\norm{\widetilde h}_{\infty,(0,q)}
   \sum_{j\ge1}\frac1{(j+1)^2}\\
 &\le \Var_{(0,q)}(\widetilde h)
 +\norm{\widetilde h}_{\infty,(0,q)}
   \sum_{j\ge1}\frac1{(j+1)^2}.
\end{align*}
Thus \eqref{eq:B-explicit} converges absolutely in $BV((q,\infty))$,
uniformly in the unimodular phase.  Its right trace satisfies
\begin{equation*}
 |B_{\eta,q}h(q+)|
 \le\norm{\widetilde h}_{\infty,(0,q)}
 \sum_{j\ge1}\frac1{(j+1)^2}.
\end{equation*}
Consequently,
\begin{align*}
 \Var_{\R_+}\big(\widetilde h+B_{\eta,q}h\big)
 &\le \Var_{(0,q)}(\widetilde h)
 +\Var_{(q,\infty)}(B_{\eta,q}h)\\
 &\quad+|\widetilde h(q-)-B_{\eta,q}h(q+)|.
\end{align*}
The preceding bounds control the right-hand side by the asserted
piecewise-$BV$ norm, including the jump at $q$.
\end{proof}

The tail has now been eliminated exactly.  What remains is to separate
the first periodization into the part propagated from the unknown interior
datum and the part forced by the freely chosen middle datum.  For
$u\in L^1((0,1))$, extended by zero, the identity
\begin{equation*}
 \cL_{\omega,q}u=S_\omega J_qu.
\end{equation*}
shows that one reciprocal round trip on the interior is represented by
$Q_{q,\zeta,\eta}$, whereas every term originating from the middle
interval is collected in $A_{q,\zeta,\eta}$.  This is the reason for the
definitions
\begin{align}
 Q_{q,\zeta,\eta}
 &:=\cL_{\zeta,q}\cL_{\eta,q},\label{eq:Q-def}\\
 A_{q,\zeta,\eta}v
 &:=-S_\zeta v+\cL_{\zeta,q}(S_\eta J_qv).
 \label{eq:A-def}
\end{align}
Substitution of \eqref{eq:B-def} into the first equation in
\eqref{eq:short-periods} yields
\begin{equation}\label{eq:interior-equation}
 (I-Q_{q,\zeta,\eta})F_0=A_{q,\zeta,\eta}v.
\end{equation}
When $\theta_1=\theta_2=0$, we have $\zeta=\eta=1$ and
$Q_{q,1,1}=\cL_{1,q}^2$.  Thus the shifts do not change the reciprocal
dynamics; they insert unimodular weights along its cylinders.
In particular, Proposition~\ref{prop:graph} identifies the
pre-annihilator with the graph of $I-Q_{q,\zeta,\eta}$ driven by the
forcing operator $A_{q,\zeta,\eta}$.

\begin{proposition}\label{prop:graph}
For $q>1$, the normalized pre-annihilator consists exactly of the
functions
\begin{equation}\label{eq:graph-formula}
 F=F_0+v+B_{\eta,q}(F_0+v),
\end{equation}
where $F_0\in L^1((0,1))$, $v\in L^1((1,q))$, and
$(F_0,v)$ satisfies \eqref{eq:interior-equation} in $L^1((0,1))$.
\end{proposition}

The untwisted specialization is
\cite[Proposition~8.5]{CHM2014}.  We retain the verification
because the exact character-weighted statement is needed here.
\begin{proof}
If $F$ satisfies the two periodizations, the second gives
\eqref{eq:B-def}; inserting that tail into the first gives
\eqref{eq:interior-equation}.  Conversely, \eqref{eq:graph-formula} makes
the second periodization an identity by $J_q^2=I$, and
\eqref{eq:interior-equation} is exactly the first one.  The equivalence
now follows from Lemma~\ref{lem:periodization}.
\end{proof}

The preceding formula becomes more informative when it is regarded as a
map between Banach spaces.  Define
\begin{equation*}
 D_{q,\zeta,\eta}(F_0,v)
 :=(I-Q_{q,\zeta,\eta})F_0-A_{q,\zeta,\eta}v
\end{equation*}
on $L^1((0,1))\oplus L^1((1,q))$, equipped with the sum norm.

\begin{proposition}\label{prop:Banach-graph}
Let $q>1$.  The map
\begin{equation}\label{eq:Phi-def}
 \Phi(F_0,v):=F_0+v+B_{\eta,q}(F_0+v)
\end{equation}
is a Banach-space isomorphism from
$\ker D_{q,\zeta,\eta}$ onto the normalized pre-annihilator.  Moreover,
\begin{equation}\label{eq:Phi-bounds}
 \norm{F_0}_1+\norm v_1
 \le \norm{\Phi(F_0,v)}_1
 \le 2\big(\norm{F_0}_1+\norm v_1\big).
\end{equation}
\end{proposition}

\begin{proof}
The operators $Q$ and $A$ are bounded on the indicated $L^1$
spaces: indeed, the contraction estimates for $S_\omega$, $J_q$ and
$\cL_{\omega,q}$ give
\begin{equation*}
 \norm{QF_0}_1\le\norm{F_0}_1,
 \qquad
 \norm{Av}_1\le2\norm v_1.
\end{equation*}
Hence $\ker D_{q,\zeta,\eta}$ is closed.  Proposition~\ref{prop:graph}
shows that $\Phi$ maps this kernel bijectively onto the normalized
pre-annihilator.  Finally, $F_0+v$ and
$B_{\eta,q}(F_0+v)$ have disjoint supports.  By
\eqref{eq:B-bound},
\begin{align*}
 \norm{\Phi(F_0,v)}_1
 &=\norm{F_0}_1+\norm v_1
   +\norm{B_{\eta,q}(F_0+v)}_1\\
 &\le2\big(\norm{F_0}_1+\norm v_1\big),
\end{align*}
and the lower bound is immediate.  Thus $\Phi$ and its inverse are
bounded.
\end{proof}

\begin{lemma}\label{lem:A-BV}
For fixed $q>1$ and $\zeta,\eta\in\T$,
\begin{equation*}
 A_{q,\zeta,\eta}:BV((1,q))\longrightarrow BV((0,1))
\end{equation*}
is bounded.
\end{lemma}

\begin{proof}
Extend $v$ by zero.  In $S_\zeta v$, only finitely many translates
can be nonzero.  On each translated subinterval, variation is bounded by
$\Var_{(1,q)}v$, while the zero-extension jumps are controlled by
\eqref{eq:BV-sup}.  Thus
\begin{equation*}
 \norm{S_\zeta v}_{BV((0,1))}
 \le C_q\norm v_{BV((1,q))}.
\end{equation*}
The map $t\mapsto q/t$ is a smooth diffeomorphism of $(1,q)$ onto
itself, and its weight $q/t^2$ has bounded derivative there.  The
product and composition rules for $BV$ therefore give
\begin{equation*}
 \norm{J_qv}_{BV((1,q))}
 \le C_q\norm v_{BV((1,q))}.
\end{equation*}
Applying the same finite-translate estimate to $S_\eta J_qv$, and then
the $BV$-boundedness of $\cL_{\zeta,q}$ proved in
Theorem~\ref{thm:LY}, proves the assertion.
\end{proof}

\section{A twisted Lasota--Yorke inequality}

The unshifted HUP analysis in
\cite[Theorem~C]{CHM2014} places the Gauss-type Perron--Frobenius operator on
$BV$.  The strong--weak variation mechanism goes back to
\cite[Theorem~1]{LasotaYorke1973}; the
countable-branch framework is developed in
\cite[Theorem~1]{Rychlik1983}.  Here the
graph equation contains the two-step
operator $Q_{q,\zeta,\eta}$ with arbitrary unimodular cylinder weights.
Its $L^1$ contraction alone does not make $I-Q_{q,\zeta,\eta}$ Fredholm.
We need a strict contraction of the variation, uniform in both phases,
while retaining every endpoint jump created by zero extension.  This is
the purpose of the twisted inequality below.

Put $I=(0,1)$, $m=\lfloor q\rfloor$ and $a=q-m$.  Modulo endpoints,
the monotonicity intervals of $T_q$ are
\begin{equation*}
 I_m=\left(\frac q{m+1},1\right),\qquad
 I_n=\left(\frac q{n+1},\frac qn\right),\quad n\ge m+1.
\end{equation*}
If $a=0$, $T_q(I_m)=I$; if $a>0$,
$T_q(I_m)=(a,1)$.  Every branch with $n\ge m+1$ maps onto $I$.

Let $\cP_r$ be the collection of maximal open cylinders of $T_q^r$.
For $C\in\cP_r$, put
\begin{equation*}
 J_C=T_q^r(C),\qquad
 \psi_C=(T_q^r|_C)^{-1}:J_C\to C,\qquad
 h_C=|\psi_C'|.
\end{equation*}

\begin{lemma}\label{lem:finite-images}
For every fixed $r$, the family
$\{J_C:C\in\cP_r\}$ is finite.  Consequently,
\begin{equation}\label{eq:delta-r}
 \delta_r:=\min_{C\in\cP_r}|J_C|>0.
\end{equation}
No lower bound uniform in $r$ is asserted or needed.
\end{lemma}

\begin{proof}
At rank one the only images are $I$ and, when $a>0$, $(a,1)$.
Suppose the rank-$r$ image family is finite and let $K$ be one of its
intervals.  The next refinement uses the intersections $K\cap I_n$.
Since the intervals $I_n$ are consecutive, all but at most two nonempty
intersections are whole branch intervals.  Their images are $I$ or
$(a,1)$.  Only the branches cut by the two endpoints of $K$ can
produce new images.  Thus each $K$ gives finitely many new image
intervals.  Induction proves finiteness.  Every nonempty cylinder image
has positive length, which proves \eqref{eq:delta-r}.
\end{proof}

\begin{lemma}\label{lem:distortion}
For every $r\ge1$ and $C\in\cP_r$,
\begin{equation*}
 \sup_{J_C}h_C\le q^{-r},\qquad
 \frac{\sup_{J_C}h_C}{\inf_{J_C}h_C}\le4.
\end{equation*}
Moreover, $h_C$ is monotone on $J_C$.
\end{lemma}

\begin{proof}
An inverse cylinder branch is a composition of maps
$\phi_n(x)=q/(n+x)$, hence is fractional-linear:
\begin{equation*}
 \psi_C(x)=\frac{A_Cx+B_C}{C_Cx+D_C}.
\end{equation*}
The coefficients may be chosen nonnegative with
\begin{equation*}
 B_C\ge A_C,\qquad D_C\ge C_C,\qquad
 |A_CD_C-B_CC_C|=q^r.
\end{equation*}
For one branch the coefficient quadruple is $(0,q,1,n)$.  Left
composition by $\phi_n$ sends
\begin{equation*}
 (A,B,C,D)\longmapsto(qC,qD,nC+A,nD+B),
\end{equation*}
which preserves the two inequalities.  It follows that
\begin{equation*}
 h_C(x)=\frac{q^r}{(C_Cx+D_C)^2}
\end{equation*}
is monotone.  Since $J_C\subset(0,1)$,
\begin{equation*}
 \frac{\sup h_C}{\inf h_C}
 \le\left(\frac{C_C+D_C}{D_C}\right)^2\le4.
\end{equation*}
Finally, $\phi_n(x)<1$ on its domain, so
\begin{equation*}
 |\phi_n'(x)|=\frac{\phi_n(x)^2}{q}\le q^{-1}.
\end{equation*}
The chain rule gives $\sup h_C\le q^{-r}$.
\end{proof}

Let
$\boldsymbol\omega=(\omega_0,\ldots,\omega_{r-1})\in\T^r$, and define
\begin{equation*}
 \cL_{\boldsymbol\omega}^{(r)}
 =\cL_{\omega_{r-1},q}\cdots\cL_{\omega_0,q}.
\end{equation*}
The digit word, hence the phase, is constant on every cylinder.  Thus
\begin{equation*}
 \cL_{\boldsymbol\omega}^{(r)}u
 =\sum_{C\in\cP_r}
 c_C\one_{J_C}h_C(u\circ\psi_C),\qquad |c_C|=1.
\end{equation*}
The phases may create jumps at the endpoints of $J_C$.  We count those
jumps by extending every summand by zero; no cancellation is used.

\begin{theorem}\label{thm:LY}
For every $r\ge1$, every
$\boldsymbol\omega\in\T^r$, and every $u\in BV(I)$,
\begin{equation}\label{eq:LY-word}
 \Var\!\left(\cL_{\boldsymbol\omega}^{(r)}u\right)
 \le4q^{-r}\Var(u)+\frac8{\delta_r}\norm u_1,
\end{equation}
and
\begin{equation*}
 \norm{\cL_{\boldsymbol\omega}^{(r)}u}_1\le\norm u_1.
\end{equation*}
Consequently,
\begin{equation}\label{eq:LY-Q}
 \Var(Q_{q,\zeta,\eta}^{\,N}u)
 \le4q^{-2N}\Var(u)+\frac8{\delta_{2N}}\norm u_1.
\end{equation}
In particular,
\begin{equation*}
 r_{\ess}(Q_{q,\zeta,\eta};BV(I))\le q^{-2}<1,\qquad
 r_{\ess}(\cL_{\omega,q};BV(I))\le q^{-1}.
\end{equation*}
\end{theorem}

\begin{proof}
Fix a $BV$ representative with one-sided traces.  If
$J=(a,b)\subset I$ and $w\in BV(J)$, then zero extension gives
\begin{equation}\label{eq:zero-extension-traces}
 \Var_I(\one_Jw)
 \le \Var_J(w)+|w(a+)|+|w(b-)|.
\end{equation}
When an endpoint of $J$ is also an endpoint of $I$, the corresponding
jump is omitted.  This formula records the endpoint jumps before any
cancellation between different cylinder phases can occur.

For $C\in\cP_r$, set
\begin{equation*}
 u_C=\frac1{|C|}\int_Cu(y)\,\dd y,\qquad v_C=u-u_C.
\end{equation*}
For almost every $x\in C$,
\begin{equation*}
 |v_C(x)|
 \le\frac1{|C|}\int_C|u(x)-u(y)|\,\dd y
 \le\Var_C(u),
\end{equation*}
and hence $\norm{v_C}_{\infty,C}\le\Var_C(u)$.  We split each cylinder
summand as
\begin{equation*}
 c_C\one_{J_C}h_C(u\circ\psi_C)
 =c_C\one_{J_C}h_C(v_C\circ\psi_C)
  +c_Cu_C\one_{J_C}h_C.
\end{equation*}
Since $|c_C|=1$, the phase does not affect either variation estimate.
The product rule, \eqref{eq:zero-extension-traces}, and the monotonicity of
$h_C$ give
\begin{align}
 \Var_I\!\left(\one_{J_C}h_C(v_C\circ\psi_C)\right)
 &\le(\sup_{J_C}h_C)\Var_C(u)
  +\norm{v_C}_{\infty,C}\Var_{J_C}(h_C)\notag\\
 &\qquad
  +2(\sup_{J_C}h_C)\norm{v_C}_{\infty,C}\notag\\
 &\le4(\sup_{J_C}h_C)\Var_C(u).
 \label{eq:residual-variation}
\end{align}
Here $\Var_{J_C}(h_C)\le\sup_{J_C}h_C$, because $h_C$ is positive
and monotone.  For the constant part, positivity and monotonicity imply
that the variation of the zero extension of $h_C$ is at most
$2\sup_{J_C}h_C$: the interior variation and the two endpoint traces
sum to twice the larger endpoint value.  Thus
\begin{equation}\label{eq:constant-variation}
 \Var_I(u_C\one_{J_C}h_C)
 \le2|u_C|\sup_{J_C}h_C.
\end{equation}
The distortion bound and \eqref{eq:delta-r} yield
\begin{equation}\label{eq:cylinder-length}
 |C|=\int_{J_C}h_C(x)\,\dd x
 \ge |J_C|\inf_{J_C}h_C
 \ge\frac{\delta_r}{4}\sup_{J_C}h_C.
\end{equation}
Since $|u_C|\le |C|^{-1}\int_C|u(y)|\,\dd y$,
\eqref{eq:constant-variation}--\eqref{eq:cylinder-length} give
\begin{equation}\label{eq:weak-cylinder}
 \Var_I(u_C\one_{J_C}h_C)
 \le\frac8{\delta_r}\int_C|u(y)|\,\dd y.
\end{equation}

The rank-$r$ cylinders are pairwise disjoint modulo endpoints.  Hence
\begin{equation*}
 \sum_{C\in\cP_r}\Var_C(u)\le\Var_I(u),\qquad
 \sum_{C\in\cP_r}\int_C|u(y)|\,\dd y=\norm u_1.
\end{equation*}
Equations \eqref{eq:residual-variation} and
\eqref{eq:weak-cylinder}, together with
$\sup_{J_C}h_C\le q^{-r}$, show that the sum of the variations of the
zero-extended cylinder summands is finite.  Moreover, change of variables
on every cylinder gives
\begin{equation*}
 \sum_{C\in\cP_r}
 \norm{\one_{J_C}h_C(u\circ\psi_C)}_1
 =\sum_{C\in\cP_r}\int_C|u(y)|\,\dd y
 =\norm u_1.
\end{equation*}
Thus the cylinder series converges absolutely in $BV(I)$.  We may sum
the estimates without using cancellation between phases or between common
image endpoints, and obtain \eqref{eq:LY-word}.  The same change of
variables and the triangle inequality give the $L^1$ contraction.

For $Q^N$, take $r=2N$ and the alternating word
$(\eta,\zeta,\ldots,\eta,\zeta)$.  This proves \eqref{eq:LY-Q}.  Put
\begin{equation*}
 a_N=4q^{-2N},\qquad C_N=1+\frac8{\delta_{2N}}.
\end{equation*}
Combining \eqref{eq:LY-Q} with the $L^1$ contraction gives the strong
and weak estimates
\begin{equation*}
 \norm{Q^Nu}_{BV}
 \le a_N\norm u_{BV}+C_N\norm u_1,
 \qquad
 \norm{Q^Nu}_1\le\norm u_1.
\end{equation*}
For every sufficiently large $N$, we have $a_N<1$.  Apply
Theorem~\ref{thm:Hennion} to $L=Q^N$, with
$\mathsf B=BV(I)$ and $\mathsf B_w=L^1(I)$.  All its assumptions are
now explicit: the first estimate makes $Q^N$ bounded on $BV(I)$, the
second makes it a contraction on $L^1(I)$, the two actions agree on
$BV(I)$, and $BV(I)\hookrightarrow L^1(I)$ is compact.  We obtain
\begin{equation*}
 r_{\ess}(Q^N;BV(I))\le a_N=4q^{-2N}.
\end{equation*}
This application is made separately for each sufficiently large fixed
$N$; no estimate on
$C_N$ uniform in $N$ is required.  Since the essential spectral radius
is the spectral radius in the Calkin algebra, spectral mapping gives
\begin{equation*}
 r_{\ess}(Q;BV(I))^N=r_{\ess}(Q^N;BV(I))\le4q^{-2N}.
\end{equation*}
Taking $N$-th roots and then letting $N\to\infty$ yields
$r_{\ess}(Q;BV(I))\le q^{-2}$.  For
$\cL_{\omega,q}^N$, use \eqref{eq:LY-word} with the constant word
$(\omega,\ldots,\omega)$.  The same argument, now with strong
coefficient $4q^{-N}$, gives
$r_{\ess}(\cL_{\omega,q};BV(I))\le q^{-1}$.
\end{proof}

\begin{remark}
The weak constant $8/\delta_{2N}$ may depend on $N$.  Only the strong
coefficient $4q^{-2N}$ has the iterate-uniform exponential form needed
for the essential spectral-radius estimate.  This is why the finite-image
lemma is an essential part of the proof.
\end{remark}

\section{Supercritical infinite-dimensionality}

Throughout this section, let
\begin{equation*}
 \mathsf B=BV((0,1)),\qquad P=\cL_{1,q},\qquad
 Q=Q_{q,\zeta,\eta},\qquad A=A_{q,\zeta,\eta},
\end{equation*}
and put $d_q(x)=\lfloor q/x\rfloor$ away from the endpoints of the
fundamental intervals.

We first prove the dimension assertion of Theorem~\ref{thm:intro-main}.
By Proposition~\ref{prop:graph}, it is enough to solve
\begin{equation*}
 (I-Q_{q,\zeta,\eta})u=A_{q,\zeta,\eta}v
\end{equation*}
for every $v\in BV((1,q))$.  The Lasota--Yorke estimate makes $I-Q$
Fredholm on $BV((0,1))$; it remains to determine its kernel and verify
the untwisted compatibility condition.

\begin{proposition}\label{prop:Fredholm}
If $q>1$, then $I-Q:\mathsf B\to\mathsf B$ is Fredholm of index zero.
\end{proposition}

\begin{proof}
By Theorem~\ref{thm:LY}, the essential spectral radius of $Q$ is less
than one.  Hence $I-sQ$ is Fredholm for $0\le s\le1$, because
\begin{equation*}
 r_{\ess}(sQ)=s\,r_{\ess}(Q)<1.
\end{equation*}
The Fredholm index is constant along this path and equals the index of the
identity at $s=0$, namely zero.
\end{proof}

The interval-growth argument in \cite[Lemma~7.1]{CHM2014} is refined
below to a single full branch, which transports one coherent phase across
the state space.

\begin{lemma}\label{lem:full-cylinder}
For every nonempty open interval $J\subset(0,1)$, there are an integer
$r\ge1$ and an interval $C\subset J$ such that
\begin{equation*}
 T_q^r:C\longrightarrow(0,1)
\end{equation*}
is a diffeomorphism.
\end{lemma}

\begin{proof}
The proof refines the interval-growth device of
\cite[Lemma~7.1]{CHM2014}.  Let
\begin{equation*}
 \mathcal D_q=\{q/n:n\in\Z,\ n>q\}
\end{equation*}
be the discontinuity set.  If an open interval $K$ contains
$q/n\in\mathcal D_q$, choose $\varepsilon>0$ so that
\begin{equation*}
 K'=\left(\frac q{n+\varepsilon},\frac qn\right)
\end{equation*}
lies in $K$.  The map $T_q$ sends $K'$ diffeomorphically onto
$(0,\varepsilon)$.  A full branch
\begin{equation*}
 I_N=\left(\frac q{N+1},\frac qN\right)
\end{equation*}
lies in $(0,\varepsilon)$ for large $N$, so its pullback through
$T_q|_{K'}$ is a rank-two full cylinder inside $K$.

Starting with $J_0=J$, continue with $J_{k+1}=T_q(J_k)$ while $J_k$
avoids $\mathcal D_q$.  On every such continuity interval,
\begin{equation*}
 |T_q'(x)|=q/x^2\ge q>1,
\end{equation*}
and hence
\begin{equation*}
 |J_k|\ge q^k|J|
\end{equation*}
as long as the construction continues.  During that time
$J_k\subset(0,1)$, and hence $|J_k|\le1$.  Since $q>1$, the displayed
lower bound exceeds $1$ for all sufficiently large $k$, a contradiction.
Thus some $J_k$ meets $\mathcal D_q$.  Apply the preceding rank-two
construction there and pull it back along the earlier continuity
branches.
\end{proof}

The Lasota--Yorke inequality supplies the following $L^1$-to-$BV$
regularization for unit-modulus eigenvectors.  This bridge is needed because
quasi-compactness on $BV$ alone says nothing about a general $L^1$
vector.  For the positive operator it follows from
\cite[Theorem~C]{CHM2014} and the spectral results in
\cite[Theorem~2]{Rychlik1983}.  We retain the argument because it also
applies to the complex-weighted twisted operator $Q$, which is generally
not positivity-preserving.

\begin{lemma}\label{lem:peripheral-BV}
Let $L=P$, or let $L=Q_{q,\zeta,\eta}$.  If $|\lambda|=1$, then
\begin{equation*}
 \ker(L-\lambda I:L^1((0,1))\to L^1((0,1)))\subset BV((0,1)).
\end{equation*}
\end{lemma}

\begin{proof}
By Theorem~\ref{thm:LY}, there are an integer $N\ge1$, constants
$0<a<1$ and $C<\infty$, and $R=L^N$, such that
\begin{equation*}
 \norm{Ru}_{BV}\le a\norm u_{BV}+C\norm u_1,
 \qquad \norm{Ru}_1\le\norm u_1.
\end{equation*}
Indeed, one uses the constant word of length $N$ for $P$, and the
alternating word of length $2N$ for $Q$.  Iteration gives
\begin{equation}\label{eq:peripheral-iterate}
 \norm{R^ku}_{BV}
 \le a^k\norm u_{BV}+\frac{C}{1-a}\norm u_1.
\end{equation}
Suppose that $Lh=\lambda h$ in $L^1((0,1))$, and put $\mu=\lambda^N$.
Choose $g_j\in BV((0,1))$ with $g_j\to h$ in $L^1((0,1))$, and choose
$k_j\ge j$ so that $a^{k_j}\norm{g_j}_{BV}\le1$.  The
phase-corrected sequence
\begin{equation*}
 u_j=\mu^{-k_j}R^{k_j}g_j
\end{equation*}
is uniformly bounded in $BV((0,1))$ by \eqref{eq:peripheral-iterate}, while
\begin{equation*}
 \norm{u_j-h}_1
 =\norm{R^{k_j}g_j-R^{k_j}h}_1
 \le\norm{g_j-h}_1\longrightarrow0.
\end{equation*}
Compactness of $BV((0,1))\hookrightarrow L^1((0,1))$, together with lower
semicontinuity of variation, gives $h\in BV((0,1))$.
\end{proof}

\begin{lemma}\label{lem:equality-phase}
Let $r\ge1$, let $\chi:(0,1)\to\T$ be measurable, and define
\begin{equation*}
 (L_{\chi,r}h)(x)
 =
 \sum_{T_q^ry=x}
 \frac{\chi(y)}{|(T_q^r)'(y)|}h(y).
\end{equation*}
Suppose that $h\in L^1((0,1))$, $|\lambda|=1$, and
\begin{equation*}
 L_{\chi,r}h=\lambda h,\qquad
 P^r|h|=|h|>0\quad\mathrm{a.e.}
\end{equation*}
If $\phi=h/|h|$, then
\begin{equation*}
 \lambda\phi(T_q^ry)=\chi(y)\phi(y)
 \quad\mathrm{a.e.}\ y\in(0,1).
\end{equation*}
\end{lemma}

\begin{proof}
For almost every $x$, both inverse-branch sums converge absolutely and
\begin{equation*}
 \left|
 \sum_{T_q^ry=x}
 \frac{\chi(y)}{|(T_q^r)'(y)|}h(y)
 \right|
 =
 |h(x)|
 =
 \sum_{T_q^ry=x}
 \frac{|h(y)|}{|(T_q^r)'(y)|}.
\end{equation*}
For an absolutely summable family $(z_j)$ satisfying
\begin{equation*}
 \left|\sum_jz_j\right|=\sum_j|z_j|>0,
\end{equation*}
all nonzero $z_j$ have the argument of the sum.  Indeed, if
\begin{equation*}
 \vartheta=\frac{\sum_jz_j}{|\sum_jz_j|},
\end{equation*}
then
\begin{equation*}
 \sum_j\left(|z_j|-\operatorname{Re}
 (\overline\vartheta z_j)\right)=0,
\end{equation*}
and every summand is nonnegative.  Applying this observation at the
target point $x$ gives
\begin{equation*}
 \chi(y)\phi(y)=\lambda\phi(x)
\end{equation*}
for every nonzero branch contribution.  On each continuity cylinder,
$T_q^r$ is a nonsingular $C^1$ diffeomorphism.  Pulling back the
exceptional target set, adjoining the null set where $h=0$, and taking
the countable union over the cylinders proves the assertion.
\end{proof}

We now isolate the positive spectral input needed for the phase analysis.
The strategy follows the HUP argument in
\cite[Theorem~7.2]{CHM2014}: a Lasota--Yorke estimate first
produces a $BV$ invariant density, and interval covering then yields
positivity and peripheral rigidity.  We include the short argument because
the present positive-branch map has a truncated edge branch, and the
twisted equality case below requires the explicit lower bound for its
density.

\begin{lemma}\label{lem:positive-spectrum}
For every $q>1$, there is a unique normalized invariant density
$\rho_q\in BV((0,1))$ such that
\begin{equation*}
 P\rho_q=\rho_q,\qquad \rho_q>0\ \mathrm{a.e.},\qquad
 \int_0^1\rho_q(x)\,\dd x=1.
\end{equation*}
Moreover,
\begin{equation}\label{eq:rho-lower}
 \operatorname*{ess\,inf}_{(0,1)}\rho_q>0,
 \qquad \sigma_{\mathrm p}(P;L^1((0,1)))\cap\T=\{1\},
 \qquad
 \ker(I-P^2:L^1((0,1))\to L^1((0,1)))=\C\rho_q.
\end{equation}
\end{lemma}

\begin{proof}
After assigning arbitrary values to $T_q$ at $0$ and at the countable
branch endpoints, view it as a map of $[0,1]$ into itself.  These
assignments do not affect its Perron--Frobenius operator.  Put
$m=\lfloor q\rfloor$ and $a=q-m=\{q\}_1$.  Modulo the branch endpoints,
the intervals
\begin{equation*}
 I_m=\left(\frac q{m+1},1\right),\qquad
 I_n=\left(\frac q{n+1},\frac qn\right),\quad n\ge m+1,
\end{equation*}
are pairwise disjoint and cover $(0,1)$.  On $I_n$ one has
\begin{equation*}
 T_q(y)=\frac qy-n,
\end{equation*}
so every branch is strictly decreasing and extends to a $C^2$ map on
the closure of its interval, with nonvanishing derivative in the
interior.  Moreover,
\begin{equation*}
 T_q(I_m)=(a,1),\qquad T_q(I_n)=(0,1)\quad(n\ge m+1).
\end{equation*}
Thus every branch image has length at least
\begin{equation*}
 \delta_q=1-a>0,
\end{equation*}
and $T_q$ is partially filling in the sense of
\cite[Definition~3.2]{CHM2014}.  Finally, on every branch,
\begin{equation*}
 |T_q'(y)|=\frac q{y^2}\ge q=1+(q-1),
 \qquad
 \frac{|T_q''(y)|}{|T_q'(y)|^2}=\frac{2y}{q}\le\frac2q.
\end{equation*}
Hence the uniform-expansion hypothesis of
\cite[Theorem~C]{CHM2014} holds for the first iterate with
$\varepsilon=q-1$, and its second-derivative hypothesis holds with
$M=2/q$.  All hypotheses of that theorem are satisfied, so it supplies
a normalized nonnegative invariant density
$\rho_q\in BV((0,1))$.

Choose a continuity point $x_0$ of its precise representative such that
$\rho_q(x_0)>0$.  On some open interval $J\ni x_0$, we have
$\rho_q\ge c>0$.  By Lemma~\ref{lem:full-cylinder}, there is a full
cylinder $C\subset J$ of rank $r$.  If
$\psi:(0,1)\to C$ is its inverse branch, then $\psi$ extends $C^1$
to $[0,1]$ with nonvanishing derivative, and invariance gives
\begin{equation}\label{eq:rho-spread}
 \rho_q(x)=P^r\rho_q(x)
 \ge |\psi'(x)|\rho_q(\psi(x))
 \ge c\inf_{(0,1)}|\psi'|>0
\end{equation}
for almost every $x$.  This proves the lower bound in
\eqref{eq:rho-lower}.

We next prove simplicity.  Let $f\in BV((0,1))$ be real, $Pf=f$, and
$\int_0^1f(x)\,\dd x=0$.  If $f\ne0$, then both $f^+$ and $f^-$ are
nonzero.  Positivity and preservation of the integral give
\begin{equation*}
 |f|=|Pf|\le P|f|,
 \qquad
 \int_0^1P|f|(x)\,\dd x=\int_0^1|f(x)|\,\dd x.
\end{equation*}
Thus $P|f|=|f|$ almost everywhere, and hence
\begin{equation*}
 Pf^+=f^+,\qquad Pf^-=f^-.
\end{equation*}
Applying the full-cylinder argument \eqref{eq:rho-spread} separately to
$f^+$ and $f^-$ makes each strictly positive almost everywhere.  This
contradicts $f^+f^-=0$.  Thus every real fixed vector of mean zero
vanishes.  If $g\in BV((0,1))$ is an arbitrary real fixed vector, then,
since $\int_0^1\rho_q\,\dd x=1$,
\begin{equation*}
 g-\left(\int_0^1g(x)\,\dd x\right)\rho_q\in\ker(I-P),
 \qquad
 \int_0^1\left[
 g-\left(\int_0^1g(x)\,\dd x\right)\rho_q
 \right]\,\dd x=0.
\end{equation*}
The preceding argument therefore gives
\begin{equation*}
 g=\left(\int_0^1g(x)\,\dd x\right)\rho_q.
\end{equation*}
Separating real and imaginary parts now yields
\begin{equation*}
 \ker\big(I-P:BV((0,1))\to BV((0,1))\big)=\C\rho_q.
\end{equation*}
Lemma~\ref{lem:peripheral-BV} extends this identity to
$L^1((0,1))$.

It remains to exclude nontrivial peripheral eigenvalues.  Suppose that
$Ph=\lambda h\ne0$ in $L^1((0,1))$, where $|\lambda|=1$.  By
Lemma~\ref{lem:peripheral-BV}, $h\in BV((0,1))$, and equality in the domination
$|Ph|\le P|h|$, together with preservation of the integral, gives
$|h|=P|h|=c\rho_q$ for some $c>0$.  Put
\begin{equation*}
 \kappa=\operatorname*{ess\,inf}\rho_q>0.
\end{equation*}
After changing the precise representative on a null set, $\rho_q\ge
\kappa$, and the Lipschitz composition rule gives
\begin{equation}\label{eq:rho-reciprocal-BV}
 \rho_q^{-1}\in BV((0,1)),
 \qquad \Var(\rho_q^{-1})\le\kappa^{-2}\Var(\rho_q).
\end{equation}
Consequently,
\begin{equation*}
 \phi=\frac{h}{c\rho_q}\in BV((0,1)),
 \qquad |\phi|=1\quad\mathrm{a.e.}
\end{equation*}
Apply Lemma~\ref{lem:equality-phase} with $r=1$ and $\chi=1$.  It gives
\begin{equation}\label{eq:P-phase-cocycle}
 \phi(T_qy)=\overline\lambda\,\phi(y)
 \quad\mathrm{a.e.}\ y\in(0,1).
\end{equation}
Choose a continuity point $x_0$ of $\phi$.  Given $\varepsilon>0$,
take an interval $J\ni x_0$ on which
$|\phi-\phi(x_0)|<\varepsilon$, and choose a full cylinder
$C\subset J$ of rank $r$, with inverse branch $\psi:(0,1)\to C$.
Iteration of \eqref{eq:P-phase-cocycle} gives, for almost every $x,z$,
\begin{equation*}
 |\phi(x)-\phi(z)|
 =|\phi(\psi(x))-\phi(\psi(z))|<2\varepsilon.
\end{equation*}
Since $\varepsilon$ is arbitrary, $\phi$ is constant almost
everywhere.  The eigenvalue equation then forces $\lambda=1$.  This
proves the peripheral assertion and the uniqueness of the normalized
density.

Finally, if $P^2h=h$, then $h+Ph$ and $h-Ph$ are eigenvectors of
$P$ for $1$ and $-1$, respectively.  The peripheral assertion makes
the second vector zero and the first a multiple of $\rho_q$; hence
$h\in\C\rho_q$.  This proves the remaining identity in
\eqref{eq:rho-lower}.
\end{proof}

\medskip
\noindent
\textit{Explicit form of the invariant density.}
When $q=N\in\N$, $N\ge2$, all branches of $T_N$ are full and the
normalized density is elementary:
\begin{equation*}
 \rho_N(x)
 =\frac{1}{\log(1+N^{-1})}\frac{1}{N+x},
 \qquad 0<x<1.
\end{equation*}
This is the formula recorded in
\cite[Remark~5.5(b)]{CHM2014}.  Indeed, only the branches $n\ge N$
contribute, and
\begin{equation*}
 \frac{N}{(n+x)^2}\frac{1}{N+N/(n+x)}
 =\frac{1}{n+x}-\frac{1}{n+x+1}.
\end{equation*}
The branch sum therefore telescopes to $(N+x)^{-1}$, while
\begin{equation*}
 \int_0^1\frac{\dd x}{N+x}=\log(1+N^{-1}).
\end{equation*}
In particular,
\begin{equation*}
 \operatorname*{ess\,inf}_{(0,1)}\rho_N
 =\frac{1}{(N+1)\log(1+N^{-1})},\qquad
 \norm{\rho_N}_\infty
 =\frac{1}{N\log(1+N^{-1})},
\end{equation*}
and
\begin{equation*}
 \Var(\rho_N)
 =\frac{1}{N(N+1)\log(1+N^{-1})}.
\end{equation*}
For $N=1$, the same expression is the classical Gauss density
$((1+x)\log2)^{-1}$; see \cite[Section~3.9]{HM2020}.

For arbitrary real $q>1$, an elementary closed formula is generally
unavailable \cite[Section~5.4]{CHM2014}, but there is an exact convergent
representation.  Write
\begin{equation*}
 q=m+a,\qquad m=\lfloor q\rfloor,\qquad 0\le a<1,
\end{equation*}
and put
\begin{equation*}
 \ell_q=\log(1+q^{-1}),\qquad
 \rho_q^{(0)}(x)=\frac{1}{\ell_q(q+x)}.
\end{equation*}
A direct branch calculation gives
\begin{equation*}
 (P\rho_q^{(0)})(x)
 =\frac{1}{\ell_q}
 \begin{cases}
  (m+1+x)^{-1},&0<x<a,\\
  (m+x)^{-1},&a<x<1.
 \end{cases}
\end{equation*}
Thus the elementary trial density is fixed precisely when $a=0$.
Set
\begin{equation*}
 \Delta_q=P\rho_q^{(0)}-\rho_q^{(0)}.
\end{equation*}
Both $\rho_q^{(0)}$ and $P\rho_q^{(0)}$ have integral one, so
$\Delta_q$ belongs to
\begin{equation*}
 BV_0((0,1))
 :=
 \left\{u\in BV((0,1)):\int_0^1u(x)\,\dd x=0\right\}.
\end{equation*}
Since $P\rho_q=\rho_q$, $\int_0^1\rho_q\,\dd x=1$, and $P$ preserves
the integral, both summands in
\begin{equation*}
 BV((0,1))=\C\rho_q\oplus BV_0((0,1))
\end{equation*}
are $P$-invariant.  Thus this is a $P$-reducing decomposition.
The $L^1$ contraction excludes eigenvalues of modulus greater than one,
and Lemma~\ref{lem:positive-spectrum} excludes every peripheral
eigenvalue other than $1$.  Moreover, $1$ has no Jordan chain:
integration of $(I-P)u=c\rho_q$ gives $0=c$.  Quasi-compactness
therefore gives
\begin{equation*}
 r\big(P|_{BV_0((0,1))}\big)<1.
\end{equation*}
Consequently,
\begin{equation*}
 \rho_q
 =\rho_q^{(0)}+\sum_{j=0}^{\infty}P^j\Delta_q
 =\lim_{N\to\infty}P^N\rho_q^{(0)}
 \qquad\mathrm{in}\ BV((0,1)),
\end{equation*}
with exponential convergence.  Equivalently, for almost every $x$,
\begin{equation*}
 \rho_q(x)
 =\lim_{N\to\infty}
 \sum_{T_q^Ny=x}
 \frac{\rho_q^{(0)}(y)}{|(T_q^N)'(y)|}.
\end{equation*}
The failure of the simple rational-function formula for nonintegral
$q$ is exactly the effect of the truncated branch $n=m$.

\begin{proposition}
\label{prop:twisted-fixed}
Let $q>1$ and $|\lambda|=1$.  Then
\begin{equation}\label{eq:twisted-fixed}
 \ker(Q_{q,\zeta,\eta}-\lambda I:
 L^1((0,1))\to L^1((0,1)))
 =
 \begin{cases}
  \C\rho_q,&(\zeta,\eta,\lambda)=(1,1,1),\\
  \{0\},&\mathrm{otherwise}.
 \end{cases}
\end{equation}
\end{proposition}

\begin{proof}
Suppose that $Qh=\lambda h\ne0$.  By
Lemma~\ref{lem:peripheral-BV}, $h\in BV((0,1))$.  Pointwise domination and
preservation of the integral by $P$ give
\begin{equation*}
 |h|=|Qh|\le P^2|h|,
 \qquad \int_0^1P^2|h|(x)\,\dd x=\int_0^1|h(x)|\,\dd x.
\end{equation*}
The inequality is therefore an equality almost everywhere.  By
Lemma~\ref{lem:positive-spectrum},
\begin{equation*}
 |h|=c\rho_q
\end{equation*}
for some $c>0$.  In view of \eqref{eq:rho-lower},
\begin{equation*}
 \phi=\frac{h}{c\rho_q}\in BV((0,1)),
 \qquad |\phi|=1\quad\mathrm{a.e.}
\end{equation*}
Indeed, \eqref{eq:rho-reciprocal-BV} and the $BV$ product rule give the
explicit bound
\begin{equation*}
 \Var(\phi)
 \le \frac1c\left(\kappa^{-1}\Var(h)
 +\norm{h}_\infty\kappa^{-2}\Var(\rho_q)\right),
 \qquad \kappa=\operatorname*{ess\,inf}\rho_q.
\end{equation*}

The identity
\begin{equation*}
 \cL_{\omega,q}u=P(\omega^{d_q}u)
\end{equation*}
shows that
\begin{equation*}
 Qh(x)=\sum_{T_q^2y=x}
 \frac{\eta^{d_q(y)}\zeta^{d_q(T_qy)}}{|(T_q^2)'(y)|}h(y).
\end{equation*}
Apply Lemma~\ref{lem:equality-phase} with $r=2$ and
\begin{equation*}
 \chi(y)=\eta^{d_q(y)}\zeta^{d_q(T_qy)}.
\end{equation*}
We obtain
\begin{equation}\label{eq:phase-cocycle}
 \lambda\phi(T_q^2y)
 =\eta^{d_q(y)}\zeta^{d_q(T_qy)}\phi(y)
 \quad\mathrm{a.e.}\ y\in(0,1).
\end{equation}

For all sufficiently large integers $n,m$, the maps
\begin{equation*}
 \varphi_n(x)=\frac q{n+x},\qquad
 \psi_{n,m}=\varphi_n\circ\varphi_m
\end{equation*}
are full inverse branches.  If $y=\psi_{n,m}(x)$, then
\begin{equation*}
 d_q(y)=n,\qquad d_q(T_qy)=m,
\end{equation*}
and \eqref{eq:phase-cocycle} becomes
\begin{equation}\label{eq:phase-cylinder}
 \lambda\phi(x)=\eta^n\zeta^m\phi(\psi_{n,m}(x))
 \quad\mathrm{for\ a.e.}\ x.
\end{equation}
The increasing map $\psi_{n,m}$ is a strict contraction of $[0,1]$
into a compact subinterval of $(0,1)$.  Let $p_{n,m}$ be its fixed
point.  Let $E$ be the full-measure set on which
\eqref{eq:phase-cylinder} holds and $|\phi|=1$, and set
\begin{equation*}
 E_\infty=\bigcap_{k\ge0}\psi_{n,m}^{-k}(E).
\end{equation*}
Every iterate $\psi_{n,m}^k$ is $C^1$ and bi-Lipschitz onto its image,
so $E_\infty$ has full measure.  For $x\in E_\infty$, every point
$\psi_{n,m}^k(x)$ lies in $E$, so its $\phi$-value has modulus one
and we may iterate \eqref{eq:phase-cylinder} to obtain
\begin{equation*}
 \phi(x)=(\lambda^{-1}\eta^n\zeta^m)^k
 \phi(\psi_{n,m}^k(x)).
\end{equation*}
Choose $x>p_{n,m}$ in this full-measure set.  Then
$\psi_{n,m}^k(x)\downarrow p_{n,m}$.  The right-hand limit of the
$BV$ function $\phi$ exists there; because every term approaching it
has modulus one, the limit also has modulus one.  Hence the
sequence $(\lambda^{-1}\eta^n\zeta^m)^k$ converges to a nonzero limit.  A sequence
of powers of a unimodular number has a nonzero limit only when that number
is one.  Thus
\begin{equation*}
 \lambda=\eta^n\zeta^m
\end{equation*}
for all sufficiently large $n,m$.  Comparing consecutive values of
$n$, and then of $m$, gives $\eta=\zeta=1$, and then $\lambda=1$.
Lemma~\ref{lem:positive-spectrum} gives $h\in\C\rho_q$.  This proves
\eqref{eq:twisted-fixed}.
\end{proof}

We can now remove the abstract compatibility space from the Fredholm
construction.  Recall the $P$-invariant mean-zero space
$BV_0((0,1))$ defined above.
If $(\zeta,\eta)\ne(1,1)$,
Propositions~\ref{prop:Fredholm} and~\ref{prop:twisted-fixed} show that
$I-Q$ is boundedly invertible on
$BV$.  In the untwisted case, the same propositions show that
\begin{equation}\label{eq:range-P2}
 \Ran(I-P^2)=BV_0((0,1)).
\end{equation}
Indeed, the range is closed and has codimension one, while preservation of
the integral gives its inclusion in the codimension-one space on the
right.  If $(I-P^2)u=g\in BV_0$, then
\begin{equation*}
 u_0=u-\left(\int_0^1u(x)\,\dd x\right)\rho_q\in BV_0,
 \qquad (I-P^2)u_0=g.
\end{equation*}
Moreover, $\C\rho_q\cap BV_0=\{0\}$.  Hence the restriction of
$I-P^2$ to $BV_0$ is a bounded bijection of $BV_0$; denote its
inverse by $G_q$.

For $v\in BV((1,q))$, the untwisted forcing term has mean zero.  In fact,
\begin{equation*}
 A_{q,1,1}v=-S_1v+P(S_1J_qv),
\end{equation*}
and periodization, inversion, and $P$ preserve the relevant integrals:
\begin{align*}
 \int_0^1(S_1v)(x)\,\dd x&=\int_1^qv(x)\,\dd x,
 &\int_0^1(S_1J_qv)(x)\,\dd x
 &=\int_1^q(J_qv)(x)\,\dd x=\int_1^qv(x)\,\dd x.
\end{align*}
Thus
\begin{equation}\label{eq:A-mean-zero}
 \int_0^1(A_{q,1,1}v)(x)\,\dd x=0.
\end{equation}

\begin{proposition}
\label{prop:normal-form}
Let $q>1$.  If $(\zeta,\eta)\ne(1,1)$, define
\begin{equation*}
 \mathsf E_{q,\zeta,\eta}v
 =\Phi\big((I-Q_{q,\zeta,\eta})^{-1}A_{q,\zeta,\eta}v,v\big).
\end{equation*}
Then the restriction map
\begin{equation*}
 \operatorname{Res}_{(1,q)}:
 \cA^{BV}_{q,\zeta,\eta}\longrightarrow BV((1,q)),
 \qquad
 \operatorname{Res}_{(1,q)}F=F|_{(1,q)},
\end{equation*}
is a bijection with inverse $\mathsf E_{q,\zeta,\eta}$.  In particular,
\begin{equation*}
 \cA^{BV}_{q,\zeta,\eta}
 =\mathsf E_{q,\zeta,\eta}\big(BV((1,q))\big).
\end{equation*}
If $(\zeta,\eta)=(1,1)$, put
\begin{equation*}
 \mathsf E_{q,1,1}v=\Phi(G_qA_{q,1,1}v,v),
 \qquad \Psi_q=\Phi(\rho_q,0).
\end{equation*}
Then
\begin{equation}\label{eq:normal-form-untwisted}
 \cA^{BV}_{q,1,1}
 =\C\Psi_q\oplus \mathsf E_{q,1,1}\big(BV((1,q))\big).
\end{equation}
All extension maps above are bounded from $BV((1,q))$ to $BV(\R_+)$.
\end{proposition}

\begin{proof}
Lemmas~\ref{lem:A-BV} and~\ref{lem:tail-BV}, together with the bounded
inverses just constructed, prove boundedness.  Equations
\eqref{eq:range-P2}--\eqref{eq:A-mean-zero} show that every middle datum is
admissible in the untwisted case; invertibility of $I-Q$ gives the same
conclusion in the twisted case.  Proposition~\ref{prop:graph} then shows
that all displayed functions are pre-annihilators.

Conversely, let $F\in\cA^{BV}_{q,\zeta,\eta}$, and write
$F=\Phi(F_0,v)$.  If the phases are nontrivial, the graph equation
forces $F_0=(I-Q)^{-1}Av$.  If they are trivial, it forces
$F_0-G_qAv\in\ker(I-P^2)=\C\rho_q$.  This proves both formulas and their
uniqueness assertions.  The sum in \eqref{eq:normal-form-untwisted} is
direct because $\Psi_q$ has zero middle restriction, whereas
$(\mathsf E_{q,1,1}v)|_{(1,q)}=v$.
\end{proof}

\medskip
\noindent
\textit{The zero-middle generator on $\R_+$.}
The density $\rho_q$ itself is defined only on the dynamical interval
$(0,1)$.  The corresponding global pre-annihilator is
$\Psi_q=\Phi(\rho_q,0)$.  The zero-middle construction already occurs
for the full hyperbola in \cite[Remark~8.4]{CHM2014}, and its
positive-branch graph form is implicit in
\cite[Proposition~8.5]{CHM2014}.  In the untwisted rescaled
positive-branch setting of \cite[(4.9)]{GM2026}, the corresponding
function is $\psi_0=\varrho_0-T_\gamma^*R_2^*\varrho_0$.  Under
\begin{equation*}
 x=bt,\qquad q=\frac{\gamma}{b},\qquad
 \rho_q(t)=b\varrho_0(bt),
\end{equation*}
one has
\begin{equation*}
 b^{-1}\sigma_{\gamma,b}(bt)=T_q(t),
\end{equation*}
and the adjoint tail relation gives
\begin{equation*}
 \Psi_q(t)=b\psi_0(bt)
 \qquad\mathrm{a.e.}\ t\in\R_+.
\end{equation*}
This comparison concerns the untwisted space
$\mathcal N_{\gamma,b}^{\perp}$ of \cite{GM2026}, rather than the
shifted space $\mathcal F_{\gamma,b}^{\perp}$ used in
Corollary~\ref{cor:GM-range}.  In the present arbitrary-shift
normal-form dichotomy, $\Psi_q$ is precisely the one-dimensional
ambiguity in the untwisted case.  Thus the zero-middle structure
predates the present paper; here it is incorporated into the
arbitrary-shift normal form and, when $q$ is integral, evaluated in
closed form.  In the present normalization,
Formula~\eqref{eq:B-explicit} gives
\begin{equation*}
 \Psi_q(t)
 =
 \begin{cases}
  \rho_q(t),&0<t<1,\\
  0,&1<t<q,\\
  \displaystyle
  -\sum_{\substack{j\ge1\\ qt/(q+jt)<1}}
  \frac{q^2}{(q+jt)^2}
  \rho_q\!\left(\frac{qt}{q+jt}\right),&t>q.
 \end{cases}
\end{equation*}
Together with the convergent $BV$ representation of $\rho_q$ above,
this gives an explicit convergent representation of $\Psi_q$ for every
real $q>1$.

For $q=N\in\N$, inserting the explicit invariant density from
\cite[Remark~5.5(b)]{CHM2014} into this generator makes the tail
elementary.  If $t>N$, only $j\ge N$ contribute, and, with
$\ell_N=\log(1+N^{-1})$,
\begin{equation*}
 \frac{N^2}{(N+jt)^2}
 \rho_N\!\left(\frac{Nt}{N+jt}\right)
 =
 \frac{N}{\ell_N(N+jt)(N+(j+1)t)}.
\end{equation*}
Consequently,
\begin{align*}
 \sum_{j=N}^{\infty}
 \frac{N}{(N+jt)(N+(j+1)t)}
 &=
 \frac{N}{t}\sum_{j=N}^{\infty}
 \left(\frac{1}{N+jt}-\frac{1}{N+(j+1)t}\right)\\
 &=\frac{1}{t(1+t)}.
\end{align*}
Hence
\begin{equation*}
 \Psi_N(t)
 =\frac{1}{\log(1+N^{-1})}
 \begin{cases}
  (N+t)^{-1},&0<t<1,\\
  0,&1<t<N,\\
  -[t(1+t)]^{-1},&t>N.
 \end{cases}
\end{equation*}
Although $q=1$ lies outside our standing supercritical assumption,
setting $N=1$ recovers the classical zero-middle Gauss generator.  Since
\begin{equation*}
 \int_N^\infty\frac{\dd t}{t(1+t)}
 =\log(1+N^{-1}),
\end{equation*}
the two nonzero pieces each have $L^1$ norm one, and therefore
\begin{equation*}
 \norm{\Psi_N}_{L^1(\R_+)}=2.
\end{equation*}

As a direct consequence, Proposition~\ref{prop:normal-form} removes the
special unresolved range in \cite[Remark~1.6]{GM2026}.  Let
\begin{equation*}
 \mathcal F_{\gamma,b}^{\perp}
 =\left\{f\in L^1(\R_+):
 \begin{array}{l}
 \displaystyle\int_{\R_+}
 f(x)e^{2\pi i(m+1/b)x}\,\dd x=0,\quad m\in\Z,\\
 \displaystyle\int_{\R_+}
 f(x)e^{2\pi in\gamma/x}\,\dd x=0,\quad n\in\Z
 \end{array}\right\}.
\end{equation*}

\begin{corollary}\label{cor:GM-range}
For every integer $b\ge2$ and every $\gamma>1$,
\begin{equation*}
 \dim\mathcal F_{\gamma,b}^{\perp}=\infty.
\end{equation*}
In particular, this holds for $1<\gamma\le b$.
\end{corollary}

The range $\gamma>b$ is due to Giri \cite{Giri2020} and is also
recovered in \cite[Corollary~1.7]{GM2026}; the new range in
Corollary~\ref{cor:GM-range} is $1<\gamma\le b$.

\begin{proof}
Take $\alpha=1$, $\theta_1=2/b$, and $\theta_2=0$.  Then
\begin{equation*}
 q=\gamma,
 \qquad \zeta=e^{2\pi i/b},
 \qquad \eta=1.
\end{equation*}
The middle restriction of the shifted space contains $BV((1,\gamma))$
for every $\gamma>1$, independently of $b$.
\end{proof}

\section{Construction of the entire \texorpdfstring{$L^1$}{L1}
pre-annihilator}
\label{sec:L1-normal}

Define the bounded projection $\Pi=\Pi_{q,\zeta,\eta}$ from
$L^1((0,1))$ onto $\ker(I-Q)$ by
\begin{equation*}
 \Pi h=
 \begin{cases}
  0,&(\zeta,\eta)\ne(1,1),\\[1mm]
  \displaystyle\left(\int_0^1h(x)\,\dd x\right)\rho_q,
  &(\zeta,\eta)=(1,1).
 \end{cases}
\end{equation*}
In the untwisted case, its restriction to $BV((0,1))$ is the Riesz
projection associated with the simple eigenvalue $1$; for nontrivial
twists it is the zero projection.  Theorem~\ref{thm:strong-limit} shows
that $\Pi$ is the strong asymptotic projection on $L^1((0,1))$.

\begin{theorem}\label{thm:strong-limit}
For every $h\in L^1((0,1))$,
\begin{equation*}
 Q_{q,\zeta,\eta}^{\,N}h\longrightarrow\Pi h
 \qquad\mathrm{in}\ L^1((0,1)).
\end{equation*}
On $BV((0,1))$, the convergence is exponential in the $BV$ norm.
\end{theorem}

\begin{proof}
Theorem~\ref{thm:LY}, the compact embedding
$BV((0,1))\hookrightarrow L^1((0,1))$, and Hennion's theorem show that
$Q$ is quasi-compact on $BV((0,1))$, with
\begin{equation*}
 r_{\ess}(Q;BV((0,1)))<1.
\end{equation*}
Every spectral point outside the essential disk is therefore an
isolated eigenvalue of finite algebraic multiplicity.  The
$L^1$-contraction excludes eigenvalues of modulus greater than one, and
Proposition~\ref{prop:twisted-fixed} excludes eigenvalues on the unit
circle when $(\zeta,\eta)\ne(1,1)$.  Hence
\begin{equation*}
 r(Q;BV((0,1)))<1
\end{equation*}
for nontrivial twists.

In the untwisted case $Q=P^2$.  Its only peripheral eigenvalue is $1$,
and its eigenspace is $\C\rho_q$.  The eigenvalue is algebraically
simple.  Indeed, if a nontrivial Jordan chain existed, one of its
length-two terminal segments would give $u\in BV((0,1))$ and $c\ne0$
such that
\begin{equation*}
 (I-Q)u=c\rho_q.
\end{equation*}
Since $Q$ preserves the integral and
$\int_0^1\rho_q\,\dd x=1$, integration gives
\begin{equation*}
 0=\int_0^1(I-Q)u\,\dd x=c,
\end{equation*}
a contradiction.  This also excludes every longer Jordan chain, because
such a chain contains a length-two segment.

Let $\mathcal R_1$ be the Riesz projection of $Q$ at $1$ on
$BV((0,1))$.  Algebraic simplicity gives
\begin{equation*}
 \Ran\mathcal R_1=\C\rho_q.
\end{equation*}
If $\ell(h)=\int_0^1h\,\dd x$, then $\ell Q=\ell$, and hence
\begin{equation*}
 \ell\big((zI-Q)^{-1}h\big)=\frac{\ell(h)}{z-1}
\end{equation*}
for $z$ in the resolvent set.  Integrating around a small positively
oriented contour enclosing $1$ and no other spectral point yields
\begin{equation*}
 \ell(\mathcal R_1h)=\ell(h).
\end{equation*}
Since $\ell(\rho_q)=1$, it follows that
\begin{equation*}
 \mathcal R_1h
 =\left(\int_0^1h(x)\,\dd x\right)\rho_q
 =\Pi h.
\end{equation*}
The quasi-compact spectral decomposition now gives exponential
convergence on $BV((0,1))$ in both cases.

Finally, $BV((0,1))$ is dense in $L^1((0,1))$.  Approximation by a $BV$
function, together with the $L^1$ contraction of $Q$ and the boundedness
of $\Pi$, extends the convergence to every $h\in L^1((0,1))$.
\end{proof}

Put
\begin{equation*}
 X=L^1((0,1)),\qquad Y=L^1((1,q)),\qquad
 Q=Q_{q,\zeta,\eta},\qquad A=A_{q,\zeta,\eta}.
\end{equation*}
The operators $Q$, $A$, $B_{\eta,q}$ and $\Phi$ are those of
\eqref{eq:Q-def}, \eqref{eq:A-def}, \eqref{eq:B-def} and
\eqref{eq:Phi-def}.

\subsection{The Green solvability domain}\label{subsec:green-domain}

The range of $I-Q$ is not closed on $X$, as
Proposition~\ref{prop:disk-spectrum} below shows.  Thus the bounded
$BV$ inverse cannot extend to all of $X$, and the graph equation need
not be solvable for every middle datum.  We therefore isolate exactly
those data for which the normalized Green sums converge; this largest
convergence set is the maximal Green domain defined below.

\begin{definition}
For $N\ge1$ and $v\in Y$, put
\begin{equation*}
 \mathsf G_Nv=\sum_{j=0}^{N-1}Q^jAv,
\end{equation*}
and call
\begin{equation*}
 \mathfrak D_{q,\zeta,\eta}
 =\left\{v\in Y:(\mathsf G_Nv)_{N\ge1}
 \ \mathrm{converges\ in}\ X\right\}
\end{equation*}
the \emph{maximal Green domain} associated with $(Q,A)$.  For
$v\in\mathfrak D_{q,\zeta,\eta}$, set
\begin{equation*}
 \mathsf Gv=\lim_{N\to\infty}\mathsf G_Nv.
\end{equation*}
\end{definition}

Since $\Pi Q=\Pi$, the space $\ker\Pi$ is invariant under $Q$, and
$I-Q$ maps $\ker\Pi$ into itself.

\begin{theorem}\label{thm:L1-Green}
Let $q>1$.  Then
\begin{equation}\label{eq:Green-domain-range}
 \mathfrak D_{q,\zeta,\eta}
 =\left\{v\in Y:
 Av\in\Ran\big((I-Q)|_{\ker\Pi}:
 \ker\Pi\to\ker\Pi\big)\right\}.
\end{equation}
For $v\in\mathfrak D_{q,\zeta,\eta}$,
\begin{equation}\label{eq:Green-solution}
 (I-Q)\mathsf Gv=Av,\qquad \Pi\mathsf Gv=0,
\end{equation}
and every solution of $(I-Q)u=Av$ has the unique form
\begin{equation}\label{eq:all-L1-solutions}
 u=\mathsf Gv+k,\qquad k\in\Ran\Pi=\ker(I-Q).
\end{equation}
The space $\mathfrak D_{q,\zeta,\eta}$ is Banach for the graph norm
\begin{equation*}
 \norm v_{\mathfrak D}=\norm v_1+\norm{\mathsf Gv}_1,
\end{equation*}
and contains $BV((1,q))$ densely in the ambient $L^1$ norm.
\end{theorem}

\begin{proof}
In the untwisted case, \eqref{eq:A-mean-zero} was proved first for
$v\in BV((1,q))$.  It extends to every $v\in Y$ because $BV((1,q))$ is
dense in $Y$, the map $A:Y\to X$ is bounded, and integration on $(0,1)$
is continuous.  Hence $\Pi Av=0$ in both phase regimes.
Since $\Pi Q=\Pi$, every $\mathsf G_Nv$ belongs to $\ker\Pi$.
The identities
\begin{align*}
 (I-Q)\mathsf G_Nv&=Av-Q^NAv,\\
 \mathsf G_{N+1}v-\mathsf G_Nv&=Q^NAv
\end{align*}
show that convergence of $\mathsf G_Nv$ implies
\eqref{eq:Green-solution}.  Conversely, if $(I-Q)u=Av$, then
\begin{equation}\label{eq:solution-telescope}
 \mathsf G_Nv=u-Q^Nu\longrightarrow u-\Pi u
\end{equation}
by Theorem~\ref{thm:strong-limit}.  This proves
\eqref{eq:Green-domain-range} and \eqref{eq:all-L1-solutions}.

If $v_j\to v$ in $Y$ and $\mathsf Gv_j\to u$ in $X$, passage to the
limit in \eqref{eq:Green-solution} gives
\begin{equation*}
 (I-Q)u=Av,\qquad \Pi u=0.
\end{equation*}
The preceding argument yields $u=\mathsf Gv$, so $\mathsf G$ is closed
and the graph norm is complete.  Proposition~\ref{prop:normal-form}
shows that $BV((1,q))\subset\mathfrak D_{q,\zeta,\eta}$; its ordinary
$L^1$ density gives the last assertion about the domain.

For $v\in BV((1,q))$, uniqueness of the solution in $\ker\Pi$ gives
\begin{equation*}
 \mathsf Gv=
 \begin{cases}
  (I-Q)^{-1}Av,&(\zeta,\eta)\ne(1,1),\\[1mm]
  G_qAv,&(\zeta,\eta)=(1,1).
 \end{cases}
\end{equation*}
Thus $\mathsf G$ restricts to the $BV$ inverses used in
Proposition~\ref{prop:normal-form}.
\end{proof}

\begin{corollary}\label{cor:L1-normal}
For $v\in\mathfrak D_{q,\zeta,\eta}$, extend the notation of
Proposition~\ref{prop:normal-form} by
\begin{equation*}
 \mathsf E_{q,\zeta,\eta}v=\Phi(\mathsf Gv,v).
\end{equation*}
Then
\begin{equation}\label{eq:full-normal-form}
 \cA_{q,\zeta,\eta}
 =
 \begin{cases}
  \mathsf E_{q,\zeta,\eta}
  \big(\mathfrak D_{q,\zeta,\eta}\big),
  &(\zeta,\eta)\ne(1,1),\\[1mm]
  \C\Psi_q\oplus
  \mathsf E_{q,1,1}\big(\mathfrak D_{q,1,1}\big),
  &(\zeta,\eta)=(1,1),
 \end{cases}
\end{equation}
where $\Psi_q=\Phi(\rho_q,0)$.  Moreover,
\begin{equation}\label{eq:Green-extension-bounds}
 \norm v_{\mathfrak D}
 \le\norm{\mathsf E_{q,\zeta,\eta}v}_{L^1(\R_+)}
 \le2\norm v_{\mathfrak D}.
\end{equation}
\end{corollary}

\begin{proof}
Proposition~\ref{prop:graph} and \eqref{eq:all-L1-solutions} give
\eqref{eq:full-normal-form}.
Directness in the untwisted case follows by restriction to $(1,q)$ and
then from $\rho_q\ne0$.  The inequalities
\eqref{eq:Green-extension-bounds} are \eqref{eq:Phi-bounds} with
$u=\mathsf Gv$.
\end{proof}

\subsection{The spectral obstruction on \texorpdfstring{$L^1$}{L1}}

Closed-disk $L^1$ spectra and closed-range criteria for
Frobenius--Perron operators have classical precedents in
\cite{Ding1994,DingDuLi1994}.  Here the operator is the complex-weighted
twisted product $Q_{q,\zeta,\eta}=\cL_{\zeta,q}\cL_{\eta,q}$, which is
generally not positivity-preserving.  The next proposition
proves the disk spectrum and nonclosed range directly for this operator
and identifies the range closure required by the HUP graph equation.

\begin{proposition}\label{prop:disk-spectrum}
Let $\mathbb D=\{\lambda\in\C:|\lambda|<1\}$ and define the Fredholm
essential spectrum by
\begin{equation*}
 \sigma_{\mathrm{e,F}}(Q;X)
 :=\{\lambda\in\C:
 Q-\lambda I\ \mathrm{is\ not\ Fredholm\ on}\ X\}.
\end{equation*}
Then
\begin{equation}\label{eq:disk-spectrum}
 \sigma(Q;X)=\sigma_{\mathrm{e,F}}(Q;X)=\overline{\mathbb D}.
\end{equation}
The range $\Ran(I-Q)$ is not closed.  More precisely,
\begin{equation}\label{eq:range-closure}
 \overline{\Ran(I-Q)}=\ker\Pi
 =
 \begin{cases}
  X,&(\zeta,\eta)\ne(1,1),\\[1mm]
  \displaystyle\left\{g\in X:\int_0^1g(x)\,\dd x=0\right\},
  &(\zeta,\eta)=(1,1),
 \end{cases}
\end{equation}
and $\Ran(I-Q)$ is a proper dense subspace of the displayed space.
\end{proposition}

\begin{proof}
For an integer $n>q$, put
\begin{equation*}
 I_n=\left(\frac q{n+1},\frac qn\right)
\end{equation*}
and define
\begin{equation*}
 (R_nf)(y)=\one_{I_n}(y)\frac q{y^2}
 f\!\left(\frac qy-n\right),\qquad
 R_{\omega,n}=\omega^{-n}R_n.
\end{equation*}
A change of variables gives
\begin{equation*}
 \cL_{\omega,q}R_{\omega,n}=I,\qquad
 \norm{R_{\omega,n}f}_1=\norm f_1.
\end{equation*}
Fix $n,m>q$ and set
\begin{equation*}
 \mathscr R=R_{\eta,n}R_{\zeta,m}.
\end{equation*}
Then $Q\mathscr R=I$.  Choose distinct integers $a,b>q$ outside
$\{n,m\}$, and define
\begin{equation*}
 \iota_{a,b}g=R_{\eta,a}g-R_{\eta,b}g,\qquad g\in X.
\end{equation*}
The intervals $I_a$ and $I_b$ are disjoint, so
\begin{equation}\label{eq:branch-difference-isometry}
 \norm{\iota_{a,b}g}_1
 =\norm{R_{\eta,a}g}_1+\norm{R_{\eta,b}g}_1
 =2\norm g_1.
\end{equation}
Thus $\iota_{a,b}$ is injective, and
\begin{equation*}
 \cL_{\eta,q}\iota_{a,b}g=g-g=0,
 \qquad Q\iota_{a,b}g=0.
\end{equation*}

Fix $0\ne g\in X$ and put $h=\iota_{a,b}g$.  The functions
$\mathscr R^kh$, $k\ge0$, have pairwise
disjoint supports: $\mathscr R^kh$ is supported in the two cylinders with
digit words $(n,m)^ka$ and $(n,m)^kb$.  Since $a,b\notin\{n,m\}$, these
cylinders are pairwise disjoint as $k$ varies.  Consequently,
\begin{equation*}
 \norm{\sum_{k=0}^{N-1}\mathscr R^kh}_1=N\norm h_1,
\end{equation*}
whereas
\begin{equation*}
 (I-Q)\sum_{k=0}^{N-1}\mathscr R^kh=\mathscr R^{N-1}h.
\end{equation*}
After normalization, this gives unit vectors $u_N\in X$ such that
\begin{equation}\label{eq:approx-fixed}
 \norm{(I-Q)u_N}_1=\frac1N.
\end{equation}

For $|\lambda|<1$, the disjoint-support series
\begin{equation*}
 F_{\lambda,g}
 =\sum_{k=0}^{\infty}
 \lambda^k\mathscr R^k\iota_{a,b}g
\end{equation*}
converges in $X$ and satisfies
\begin{equation*}
 QF_{\lambda,g}=\lambda F_{\lambda,g}.
\end{equation*}
Moreover,
\begin{equation*}
 (I-\lambda\mathscr R)F_{\lambda,g}=\iota_{a,b}g.
\end{equation*}
Thus $F_{\lambda,g}=0$ implies $\iota_{a,b}g=0$, and
\eqref{eq:branch-difference-isometry} gives $g=0$.  Hence
$g\mapsto F_{\lambda,g}$ embeds the infinite-dimensional space $X$ into
$\ker(Q-\lambda I)$ for every $\lambda\in\mathbb D$.  It follows that
\begin{equation*}
 \mathbb D\subset\sigma_{\mathrm{e,F}}(Q;X).
\end{equation*}
The Fredholm operators form an open subset of the bounded operators on
$X$.  Since $\lambda\mapsto Q-\lambda I$ is norm-continuous,
$\sigma_{\mathrm{e,F}}(Q;X)$ is closed.  Finally, $Q$ is an $L^1$
contraction and every invertible operator is Fredholm, so
\begin{equation*}
 \overline{\mathbb D}
 \subset\sigma_{\mathrm{e,F}}(Q;X)
 \subset\sigma(Q;X)
 \subset\overline{\mathbb D}.
\end{equation*}
This proves \eqref{eq:disk-spectrum}.

If $\Ran(I-Q)$ were closed, the induced bijection from
$X/\ker(I-Q)$ onto $\Ran(I-Q)$ would have a bounded inverse.  Thus, for
some $C>0$,
\begin{equation}\label{eq:closed-range-quotient}
 \operatorname{dist}_{L^1}\big(u,\ker(I-Q)\big)
 \le C\norm{(I-Q)u}_1,\qquad u\in X.
\end{equation}
For nontrivial twists, $\ker(I-Q)=\{0\}$ by
Proposition~\ref{prop:twisted-fixed}, so \eqref{eq:approx-fixed}
contradicts \eqref{eq:closed-range-quotient}.  In the untwisted case the
branch right inverses preserve the integral, so the vectors $u_N$ above
have mean zero.  For every $c\in\C$,
\begin{equation*}
 |c|\le\norm{u_N-c\rho_q}_1,\qquad
 1\le\norm{u_N-c\rho_q}_1+|c|,
\end{equation*}
and therefore
\begin{equation*}
 \operatorname{dist}_{L^1}(u_N,\C\rho_q)\ge\frac12.
\end{equation*}
Together with \eqref{eq:approx-fixed}, this again contradicts
\eqref{eq:closed-range-quotient}.

It remains to prove \eqref{eq:range-closure}.  The identity
$\Pi Q=\Pi$ gives $\Ran(I-Q)\subset\ker\Pi$.  If $f\in\ker\Pi$,
Theorem~\ref{thm:strong-limit} gives $Q^Nf\to0$, and
\begin{equation*}
 f-\frac1N\sum_{k=0}^{N-1}Q^kf
 =(I-Q)\sum_{j=0}^{N-2}\frac{N-1-j}{N}Q^jf.
\end{equation*}
The left-hand side tends to $f$, proving the reverse inclusion after
closure.  Properness follows from the nonclosed-range assertion.
\end{proof}

Theorem~\ref{thm:L1-Green} shows that every middle datum is admissible
if and only if $(\mathsf G_Nv)_{N\ge1}$ converges in $X$ for every
$v\in Y$.  Pointwise convergence on $Y$ implies, by the uniform
boundedness principle,
\begin{equation*}
 \sup_{N\ge1}\norm{\mathsf G_N}_{Y\to X}<\infty.
\end{equation*}
Conversely, suppose that the last supremum is $M_0<\infty$.  Given
$v\in Y$, choose $w\in BV((1,q))$.  For $N,K\ge1$,
\begin{equation*}
 \norm{\mathsf G_Nv-\mathsf G_Kv}_1
 \le
 2M_0\norm{v-w}_1+
 \norm{\mathsf G_Nw-\mathsf G_Kw}_1.
\end{equation*}
The sequence $(\mathsf G_Nw)_{N\ge1}$ converges because
$BV((1,q))\subset\mathfrak D_{q,\zeta,\eta}$.  First choosing $w$
close to $v$ and then taking $N,K$ large shows that
$(\mathsf G_Nv)_{N\ge1}$ is Cauchy.  Consequently, every middle datum
is admissible if and only if
\begin{equation*}
 \sup_{N\ge1}
 \norm{\sum_{j=0}^{N-1}Q^jA}_{Y\to X}<\infty.
\end{equation*}

Put $\Theta=T_q^2$ and, away from endpoints, define
\begin{equation*}
 a(t)=\{t\}_1,\qquad
 c(t)=\{q/t\}_1,\qquad
 b(t)=T_q(c(t)),\qquad 1<t<q,
\end{equation*}
and, for $j,k\ge0$,
\begin{equation*}
 a_j=\Theta^j\circ a,\qquad b_k=\Theta^k\circ b.
\end{equation*}
If
\begin{equation*}
 C=\begin{pmatrix}c_{11}&c_{12}\\c_{21}&c_{22}\end{pmatrix}
\end{equation*}
is nonsingular, write
\begin{equation*}
 C\cdot x=\frac{c_{11}x+c_{12}}{c_{21}x+c_{22}}.
\end{equation*}
Then $(AB)\cdot x=A\cdot(B\cdot x)$ wherever the expressions are
defined.  Set
\begin{equation*}
 M_n(q)=
 \begin{pmatrix}-n&q\\1&0\end{pmatrix},
 \qquad
 D_r=\begin{pmatrix}1&-r\\0&1\end{pmatrix},
\end{equation*}
and, for a finite word
$\boldsymbol n=(n_1,\ldots,n_\ell)$, set
\begin{equation*}
 M_{\boldsymbol n}(q)
 =M_{n_\ell}(q)\cdots M_{n_1}(q),
 \qquad M_{\varnothing}(q)=I.
\end{equation*}

For $a_j$, let $\mathscr P(a_j)$ be the maximal open intervals on which
the finite digit vector
\begin{equation*}
 \left(\lfloor t\rfloor,d_q(a(t)),d_q(T_qa(t)),\ldots,
 d_q(T_q^{2j-1}a(t))\right)
\end{equation*}
is defined and constant, with only $\lfloor t\rfloor$ retained when
$j=0$.  Define $\mathscr P(b_k)$ in the same way from
\begin{equation*}
 \left(\left\lfloor\frac qt\right\rfloor,d_q(c(t)),
 d_q(T_qc(t)),\ldots,d_q(T_q^{2k}c(t))\right).
\end{equation*}
For such a family $\mathscr P$, let
\begin{equation*}
 \operatorname{End}(\mathscr P)
 =\bigcup_{U\in\mathscr P}\big(\partial U\cap(1,q)\big),
\end{equation*}
and put
\begin{equation}\label{eq:all-branch-endpoints}
 \mathscr B_\infty
 =\bigcup_{j\ge0}\left(
 \operatorname{End}(\mathscr P(a_j))
 \cup\operatorname{End}(\mathscr P(b_j))\right).
\end{equation}

\begin{lemma}\label{lem:branch-coding}
The set $\mathscr B_\infty$ is countable and contains every branch
endpoint encountered at any finite future iterate of the maps above.
If $U\in\mathscr P(a_j)$ and $\lfloor t\rfloor=r$ on $U$, then there is
a word $\boldsymbol n$ of length $2j$ such that
\begin{equation*}
 a_j(t)=\big(M_{\boldsymbol n}(q)D_r\big)\cdot t,\qquad t\in U.
\end{equation*}
If $U\in\mathscr P(b_k)$, then there is a word $\boldsymbol m$ of length
$2k+2$ such that
\begin{equation*}
 b_k(t)=M_{\boldsymbol m}(q)\cdot t,\qquad t\in U.
\end{equation*}
Every displayed restriction is a nonconstant M\"obius diffeomorphism.
For $j=0$, the word in the first formula is empty; the word in the
second formula is never empty.
\end{lemma}

\begin{proof}
On the branch with digit $n$,
\begin{equation*}
 T_q(x)=\frac qx-n=M_n(q)\cdot x,
\end{equation*}
whereas $a(t)=D_r\cdot t$ when $\lfloor t\rfloor=r$.  If the successive
digits along the orbit of $a(t)$ are $n_1,\ldots,n_{2j}$, then
\begin{equation*}
 T_q^{2j}(a(t))
 =\big(M_{n_{2j}}(q)\cdots M_{n_1}(q)D_r\big)\cdot t.
\end{equation*}
This proves the first formula, including the empty-word case.  If
$m_1=\lfloor q/t\rfloor$, then
\begin{equation*}
 c(t)=M_{m_1}(q)\cdot t.
\end{equation*}
Since $b_k=T_q^{2k+1}\circ c$, adjoining the next $2k+1$ digits gives
\begin{equation*}
 b_k(t)
 =\big(M_{m_{2k+2}}(q)\cdots M_{m_1}(q)\big)\cdot t.
\end{equation*}
All these matrices are nonsingular for $q>0$, so their restrictions are
nonconstant M\"obius diffeomorphisms.  At each depth, the maximal branch
intervals form a countable disjoint family.  Their endpoints are
countable, and the countable union over all depths proves the assertion
about $\mathscr B_\infty$.  Its definition includes the initial
discontinuities and the pullbacks of every later branch endpoint.
\end{proof}

\begin{proposition}
\label{prop:generic-growth}
There is a countable set $\mathfrak E\subset(1,\infty)$ of algebraic
numbers such that, whenever $q\notin\mathfrak E$, for all
$\zeta,\eta\in\T$ and $N\ge1$,
\begin{equation}\label{eq:generic-sharp-growth}
 \norm{\sum_{j=0}^{N-1}Q^jA}_{Y\to X}=2N.
\end{equation}
Consequently, $\mathfrak D_{q,\zeta,\eta}$ is proper and meagre in $Y$,
and $\mathsf G$ is unbounded with respect to the ambient $L^1$ norm.
\end{proposition}

\begin{proof}
Fix $r\ge1$ and two finite words $\boldsymbol n,\boldsymbol m$.  Put
\begin{equation*}
 A_0(q)=M_{\boldsymbol n}(q)D_r,\qquad
 B_0(q)=M_{\boldsymbol m}(q).
\end{equation*}
Since
\begin{equation*}
 \det M_n(q)=-q,\qquad \det D_r=1,
\end{equation*}
both matrices are nonsingular for $q>0$.  Hence their M\"obius
transformations agree on a nonempty open interval if and only if
\begin{equation*}
 A_0(q)=\lambda B_0(q)
\end{equation*}
for some $\lambda\ne0$.  Equivalently, after listing the four entries
of $A_0$ and $B_0$ in the same order, all six projective minors
\begin{equation*}
 \Delta_{\mu\nu}(q)
 =(A_0)_\mu(q)(B_0)_\nu(q)
 -(A_0)_\nu(q)(B_0)_\mu(q),
 \qquad 1\le\mu<\nu\le4,
\end{equation*}
vanish.  These minors belong to $\Z[q]$, and they are not all the zero
polynomial.

Indeed, if $\boldsymbol n$ is nonempty, induction gives
\begin{equation*}
 M_{\boldsymbol n}(0)
 =\begin{pmatrix}c_1&0\\c_2&0\end{pmatrix},
 \qquad (c_1,c_2)\ne(0,0),
\end{equation*}
because every branch index is a positive integer.  Consequently,
\begin{equation*}
 M_{\boldsymbol n}(0)D_r
 =\begin{pmatrix}c_1&-rc_1\\c_2&-rc_2\end{pmatrix}.
\end{equation*}
If both words are nonempty, the latter matrix has a nonzero second
column, while $M_{\boldsymbol m}(0)$ has zero second column, so the two
matrices are not projectively proportional.  If
$\boldsymbol n=\varnothing$ and $\boldsymbol m\ne\varnothing$, then
$D_r$ has rank two whereas $M_{\boldsymbol m}(0)$ has rank one.  If
$\boldsymbol n\ne\varnothing$ and $\boldsymbol m=\varnothing$, the
ranks of $M_{\boldsymbol n}(0)D_r$ and $I$ are one and two,
respectively.  Finally, if both words are empty, then $D_r$ is not a
scalar multiple of $I$ because $r\ge1$.  Thus, in every case, at least
one projective minor is nonzero at $q=0$, and therefore is a nonzero
integer polynomial.

For each triple $(r,\boldsymbol n,\boldsymbol m)$, choose one such
nonzero minor.  An identity between the corresponding branches can
occur only at a root of that polynomial.  Its roots are finite in
number and algebraic.  Taking the union of these root sets over the
countable collection of $r$ and finite branch words defines a
countable set
\begin{equation*}
 \mathfrak E\subset(1,\infty)
\end{equation*}
consisting of algebraic numbers.

Fix $q\notin\mathfrak E$.  Lemma~\ref{lem:branch-coding} and the
definition of $\mathfrak E$ show that no restriction of an $a_j$ agrees
identically with a restriction of a $b_k$ on a nonempty open interval.
Restrictions belonging to the same family cannot agree identically
either.  Indeed, if $a_j=a_k$ on a common branch interval $U$ with
$j<k$,
then a branch of $\Theta^{k-j}$ is the identity on the open interval
$a_j(U)$.  This is impossible because
\begin{equation*}
 |T_q'(x)|=\frac q{x^2}>q,\qquad 0<x<1,
\end{equation*}
so every branch of $\Theta^{k-j}$ has derivative of modulus greater than
$q^{2(k-j)}>1$.  The same argument applies to the $b$-family.

Let
\begin{equation*}
 \mathscr F=\{a_j:j\ge0\}\cup\{b_k:k\ge0\}.
\end{equation*}
For distinct $f,g\in\mathscr F$ and branch intervals
$U\in\mathscr P(f)$ and $V\in\mathscr P(g)$, the restrictions of $f$
and $g$ to $U\cap V$ are nonidentical M\"obius maps.  Hence
\begin{equation*}
 \{t\in U\cap V:f(t)=g(t)\}
\end{equation*}
is finite.  The union $\mathscr Z$ of these sets over all
$f,g,U,V$ is countable.  The set $\mathscr B_\infty$ in
\eqref{eq:all-branch-endpoints} contains the endpoints for every finite
future iterate, so we may choose
\begin{equation*}
 t_0\in(1,q)\setminus(\mathscr B_\infty\cup\mathscr Z).
\end{equation*}
Then
\begin{equation*}
 \{a_j(t_0):j\ge0\}\cup\{b_k(t_0):k\ge0\}
\end{equation*}
is pairwise distinct.

For a fixed $N$, let $U_a\in\mathscr P(a_N)$ and
$U_b\in\mathscr P(b_N)$ be the unique branch intervals containing
$t_0$.  Choose an open interval $U_N\ni t_0$ such that
\begin{equation*}
 \overline{U_N}\subset U_a\cap U_b,
\end{equation*}
and shrink it so that the closures of the $2N$ image intervals
\begin{equation*}
 a_j(U_N),\qquad b_j(U_N),\qquad 0\le j<N,
\end{equation*}
are pairwise disjoint.  Thus $U_N$ lies in a common refinement of the
two branch partitions through the next iterate.  By the definitions of
$\mathscr P(a_N)$ and $\mathscr P(b_N)$, every $a_j$ and $b_j$,
$0\le j\le N$, is a M\"obius diffeomorphism on $U_N$, and, for
$0\le j<N$, the four functions
\begin{equation*}
 d_q(a_j(t)),\quad d_q(T_qa_j(t)),\quad
 d_q(b_j(t)),\quad d_q(T_qb_j(t))
\end{equation*}
are constant there.
Choose $v_N\in C_c^\infty(U_N)$ with $v_N\ge0$ and
$\norm{v_N}_1=1$.

We record the pushforward calculation, including its phases.  Put
\begin{equation*}
 \nu(x)=\left\lfloor\frac qx\right\rfloor,\qquad 0<x<1.
\end{equation*}
After shrinking $U_N$ once more if necessary, the integers
\begin{equation*}
 r=\lfloor t\rfloor,\qquad
 m=\left\lfloor\frac qt\right\rfloor,\qquad
 n=\nu(c(t)),\qquad
 c(t)=\frac qt-m
\end{equation*}
are constant on $U_N$.  Thus $a(t)=t-r$ and
$b(t)=T_q(c(t))$ there.  If $f:U_N\to f(U_N)$ is any of the
diffeomorphisms under consideration, define
\begin{equation*}
 (f_\#v_N)(x)
 =\one_{f(U_N)}(x)
 \frac{v_N(f^{-1}(x))}
 {|f'(f^{-1}(x))|}.
\end{equation*}
Then, for every $\varphi\in L^\infty(I)$,
\begin{equation*}
 \int_I(f_\#v_N)(x)\varphi(x)\,\dd x
 =\int_{U_N}v_N(t)\varphi(f(t))\,\dd t,
 \qquad
 \norm{f_\#v_N}_1=\norm{v_N}_1=1.
\end{equation*}

The transfer-operator change of variables gives
\begin{equation}\label{eq:twisted-transfer-duality}
 \int_I(\cL_{\omega,q}h)(x)\varphi(x)\,\dd x
 =\int_Ih(y)\omega^{\nu(y)}
   \varphi(T_qy)\,\dd y.
\end{equation}
For the first term of $A$, the substitution $t=x+r$ gives
\begin{align*}
 \int_I(S_\zeta v_N)(x)\varphi(x)\,\dd x
 &=\sum_{\ell\ge1}\zeta^\ell
   \int_0^1v_N(x+\ell)\varphi(x)\,\dd x\\
 &=\zeta^r\int_{U_N}v_N(t)\varphi(a(t))\,\dd t.
\end{align*}
For the second term, \eqref{eq:twisted-transfer-duality} and the
definition of $J_q$ yield
\begin{align*}
 &\int_I\cL_{\zeta,q}(S_\eta J_qv_N)(x)\varphi(x)\,\dd x\\
 &\quad=\sum_{\ell\ge1}\eta^\ell
 \int_0^1\frac q{(\ell+y)^2}
 v_N\!\left(\frac q{\ell+y}\right)
 \zeta^{\nu(y)}\varphi(T_qy)\,\dd y\\
 &\quad=\eta^m
 \int_{U_N}v_N(t)\zeta^{\nu(c(t))}
 \varphi(T_q(c(t)))\,\dd t\\
 &\quad=\eta^m\zeta^n
 \int_{U_N}v_N(t)\varphi(b(t))\,\dd t.
\end{align*}
In the third line we used $t=q/(m+y)$ on the only summand meeting
$\operatorname{supp}v_N$; explicitly,
\begin{equation*}
 y=c(t),\qquad
 |\dd y|=\frac q{t^2}\,\dd t,\qquad
 \frac q{(m+y)^2}=\frac{t^2}{q},
\end{equation*}
so the two Jacobian factors cancel.  Therefore
\begin{equation}\label{eq:A-two-pushforwards}
 Av_N=-\zeta^r a_\#v_N+\eta^m\zeta^n b_\#v_N.
\end{equation}

Applying \eqref{eq:twisted-transfer-duality} twice also gives
\begin{equation*}
 \int_I(Qh)(x)\varphi(x)\,\dd x
 =\int_Ih(y)\eta^{\nu(y)}\zeta^{\nu(T_qy)}
   \varphi(\Theta y)\,\dd y.
\end{equation*}
If
\begin{equation*}
 f\in\{a_j:0\le j<N\}\cup\{b_j:0\le j<N\},
\end{equation*}
the functions
\begin{equation*}
 t\longmapsto\nu(f(t)),
 \qquad
 t\longmapsto\nu(T_qf(t))
\end{equation*}
are constant on $U_N$.  Hence the factor
\begin{equation*}
 \kappa_f
 =\eta^{\nu(f(t))}\zeta^{\nu(T_qf(t))}
\end{equation*}
is constant for $t\in U_N$.  Consequently, one application of $Q$
satisfies
\begin{align*}
 \int_I Q(f_\#v_N)(x)\varphi(x)\,\dd x
 &=\kappa_f\int_I(f_\#v_N)(y)
   \varphi(\Theta y)\,\dd y\\
 &=\kappa_f\int_{U_N}v_N(t)
   \varphi(\Theta(f(t)))\,\dd t,
\end{align*}
or equivalently
\begin{equation*}
 Q(f_\#v_N)=\kappa_f(\Theta\circ f)_\#v_N.
\end{equation*}

Taking an arbitrary $t\in U_N$ and interpreting an empty product as
$1$, define
\begin{align*}
 \alpha_j
 &=\zeta^r\prod_{\ell=0}^{j-1}
   \eta^{\nu(a_\ell(t))}
   \zeta^{\nu(T_qa_\ell(t))},\\
 \beta_j
 &=\eta^m\zeta^n\prod_{\ell=0}^{j-1}
   \eta^{\nu(b_\ell(t))}
   \zeta^{\nu(T_qb_\ell(t))}.
\end{align*}
The branch choices make these expressions independent of $t\in U_N$,
and $|\alpha_j|=|\beta_j|=1$.  Iterating
\eqref{eq:A-two-pushforwards} gives
\begin{equation*}
 Q^jAv_N
 =-\alpha_j(a_j)_\#v_N+\beta_j(b_j)_\#v_N,
 \qquad 0\le j<N.
\end{equation*}
All $2N$ pushforward densities have pairwise disjoint supports.
Consequently,
\begin{align*}
 \norm{\sum_{j=0}^{N-1}Q^jAv_N}_1
 &=\sum_{j=0}^{N-1}
   \left(
    |\alpha_j|\norm{(a_j)_\#v_N}_1
    +|\beta_j|\norm{(b_j)_\#v_N}_1
   \right)\\
 &=2N.
\end{align*}
Since $\norm{v_N}_Y=1$, this proves the lower bound in
\eqref{eq:generic-sharp-growth}.  The reverse bound follows from
\begin{equation*}
 \norm{\sum_{j=0}^{N-1}Q^jA}_{Y\to X}
 \le\sum_{j=0}^{N-1}\norm Q_{X\to X}^j\norm A_{Y\to X}
 \le2N.
\end{equation*}

If $\mathfrak D_{q,\zeta,\eta}$ were all of $Y$, pointwise convergence of
$\mathsf G_N$ and the uniform boundedness principle would contradict
\eqref{eq:generic-sharp-growth}.  For $m\ge1$, the set
\begin{equation*}
 C_m=\left\{v\in Y:\sup_N\norm{\mathsf G_Nv}_1\le m\right\}
\end{equation*}
is closed and absolutely convex.  If it contained a ball $B(v_0,r)$,
then $v_0\pm(r/2)w\in C_m$ for every $\norm w_1\le1$.  Subtracting the
two images gives
\begin{equation*}
 \sup_N\norm{\mathsf G_Nw}_1\le\frac{2m}{r},
\end{equation*}
contrary to \eqref{eq:generic-sharp-growth}.  Thus $C_m$ has empty
interior.  Since
$\mathfrak D_{q,\zeta,\eta}\subset\bigcup_{m\ge1}C_m$, the domain is
meagre.  Finally,
\eqref{eq:solution-telescope} gives
\begin{equation*}
 \mathsf G_Nv=\mathsf Gv-Q^N\mathsf Gv.
\end{equation*}
Suppose that $\mathsf G$ were bounded for the ambient norm, so that
\begin{equation*}
 \norm{\mathsf Gv}_1\le C\norm v_1,
 \qquad v\in\mathfrak D_{q,\zeta,\eta}.
\end{equation*}
Since $Q$ is a contraction,
\begin{equation*}
 \norm{\mathsf G_Nv}_1
 \le\norm{\mathsf Gv}_1+\norm{Q^N\mathsf Gv}_1
 \le2C\norm v_1
\end{equation*}
on this domain.  Theorem~\ref{thm:L1-Green} says that the domain is
dense in $Y$, while every finite Green sum
$\mathsf G_N:Y\to X$ is bounded.  Approximation therefore extends the
last estimate to every $v\in Y$, giving
\begin{equation*}
 \norm{\mathsf G_N}_{Y\to X}\le2C,\qquad N\ge1.
\end{equation*}
This contradicts \eqref{eq:generic-sharp-growth}; hence $\mathsf G$ is
unbounded for the ambient $L^1$ norm.
\end{proof}

\begin{remark}\label{rem:graph-core}
Let
\begin{equation*}
 \mathfrak D_{\mathrm{core}}
 =\overline{BV((1,q))}^{\,\norm{\cdot}_{\mathfrak D}}.
\end{equation*}
By \eqref{eq:Green-extension-bounds},
\begin{equation*}
 \overline{\mathsf E_{q,\zeta,\eta}(BV((1,q)))}^{\,L^1(\R_+)}
 =\mathsf E_{q,\zeta,\eta}(\mathfrak D_{\mathrm{core}}).
\end{equation*}
Consequently, the equality
$\mathfrak D_{\mathrm{core}}=\mathfrak D_{q,\zeta,\eta}$ is equivalent to
\begin{equation*}
 \begin{cases}
  \displaystyle
  \cA_{q,\zeta,\eta}
  =\overline{\mathsf E_{q,\zeta,\eta}(BV((1,q)))}^{\,L^1(\R_+)},
  &(\zeta,\eta)\ne(1,1),\\[2mm]
  \displaystyle
  \cA_{q,1,1}
  =\C\Psi_q\oplus
  \overline{\mathsf E_{q,1,1}(BV((1,q)))}^{\,L^1(\R_+)},
  &(\zeta,\eta)=(1,1).
 \end{cases}
\end{equation*}
\end{remark}

\begin{proof}[Proof of Theorem~\ref{thm:intro-main}]
Corollary~\ref{cor:L1-normal} gives the exact normal form for the entire
$L^1$ pre-annihilator.
Proposition~\ref{prop:normal-form} gives an extension of every
$v\in BV((1,q))$ and the asserted uniqueness alternative.  Since
$BV((1,q))$ is infinite-dimensional, so is the normalized
pre-annihilator.  Scaling back from $F$ to $f$ completes the proof.
\end{proof}

\medskip
\noindent
We have therefore settled the infinite-dimensionality part of
\cite[Open Problem~1.4]{GM2026} for arbitrary shifts and every $q>1$.
For its characterization part, Corollary~\ref{cor:L1-normal} provides
an exact operator-theoretic normal form in terms of the maximal Green
domain.  An intrinsic function-space description of that domain remains
open.  At the critical endpoint $q=1$,
\cite[Theorem~1.2]{GM2026} proves conditional modulus rigidity and hence
that the critical pre-annihilator has dimension at most one.
Equivalently, it is either zero-dimensional or one-dimensional, and the
cited theorem does not decide between these alternatives.  Since it does
not establish the existence of a nonzero element, no critical
nonuniqueness conclusion is drawn here.  The endpoint lies outside the
present $q>1$ method; in particular, the strict $BV$ contraction used
above is unavailable at $q=1$.

\section*{Acknowledgments}

The author thanks Professor Jingwei Guo.  This work was supported by the
National Natural Science Foundation of China (grant 12341102).

\subsection*{Availability of Data}

No data were used for the research described in the article.

\subsection*{Declarations of Conflict of Interest}

The author declares no conflict of interest.


\begin{thebibliography}{99}

\bibitem{Bagchi2018}
S.~Bagchi,
\emph{Heisenberg uniqueness pairs corresponding to a finite number of
parallel lines},
Adv. Math. \textbf{325} (2018), 814--823.
\href{https://doi.org/10.1016/j.aim.2017.12.012}
{doi:10.1016/j.aim.2017.12.012}.

\bibitem{BakanEtAl2021}
A.~Bakan, H.~Hedenmalm, A.~Montes--Rodr\'{\i}guez, D.~Radchenko and
M.~Viazovska,
\emph{Fourier uniqueness in even dimensions},
Proc. Natl. Acad. Sci. USA \textbf{118} (2021), no.~15,
Art.~e2023227118.
\href{https://doi.org/10.1073/pnas.2023227118}
{doi:10.1073/pnas.2023227118}.

\bibitem{BlasiBabot2013}
D.~Blasi Babot,
\emph{Heisenberg uniqueness pairs in the plane: three parallel lines},
Proc. Amer. Math. Soc. \textbf{141} (2013), no.~11, 3899--3904.
\href{https://doi.org/10.1090/S0002-9939-2013-11678-3}
{doi:10.1090/S0002-9939-2013-11678-3}.

\bibitem{CHM2014}
F.~Canto-Mart\'{\i}n, H.~Hedenmalm and A.~Montes--Rodr\'{\i}guez,
\emph{Perron--Frobenius operators and the Klein--Gordon equation},
J. Eur. Math. Soc. \textbf{16} (2014), no.~1, 31--66.
\href{https://doi.org/10.4171/JEMS/427}{doi:10.4171/JEMS/427}.

\bibitem{Ding1994}
J.~Ding,
\emph{A closed range theorem for the Frobenius--Perron operator and its
application to the spectral analysis},
J. Math. Anal. Appl. \textbf{184} (1994), no.~1, 156--167.
\href{https://doi.org/10.1006/jmaa.1994.1191}
{doi:10.1006/jmaa.1994.1191}.

\bibitem{DingDuLi1994}
J.~Ding, Q.~Du and T.~Y.~Li,
\emph{The spectral analysis of Frobenius--Perron operators},
J. Math. Anal. Appl. \textbf{184} (1994), no.~2, 285--301.
\href{https://doi.org/10.1006/jmaa.1994.1200}
{doi:10.1006/jmaa.1994.1200}.

\bibitem{Giri2020}
D.~K.~Giri,
\emph{Fourier nonuniqueness sets for the hyperbola and the
Perron--Frobenius operators},
arXiv:2009.09516v1 (2020).
\href{https://doi.org/10.48550/arXiv.2009.09516}
{doi:10.48550/arXiv.2009.09516}.

\bibitem{GM2024}
D.~Giri and R.~Manna,
\emph{Revisit on Heisenberg uniqueness pair for the hyperbola},
Math. Z. \textbf{306} (2024), no.~3, Art.~39.
\href{https://doi.org/10.1007/s00209-024-03443-6}
{doi:10.1007/s00209-024-03443-6}.

\bibitem{GM2026}
D.~Giri and R.~Manna,
\emph{On the Heisenberg uniqueness pairs for a branch of the hyperbola},
J. Math. Anal. Appl. \textbf{561} (2026), no.~1, Art.~130618.
\href{https://doi.org/10.1016/j.jmaa.2026.130618}
{doi:10.1016/j.jmaa.2026.130618}.

\bibitem{GR2021}
D.~Giri and R.~Rawat,
\emph{Heisenberg uniqueness pairs for the hyperbola},
Bull. Lond. Math. Soc. \textbf{53} (2021), no.~1, 16--25.
\href{https://doi.org/10.1112/blms.12391}{doi:10.1112/blms.12391}.

\bibitem{GRCorr2022}
D.~Giri and R.~Rawat,
\emph{Corrigendum: Heisenberg uniqueness pairs for the hyperbola},
Bull. Lond. Math. Soc. \textbf{54} (2022), no.~5, 2041--2043.
\href{https://doi.org/10.1112/blms.12730}{doi:10.1112/blms.12730}.

\bibitem{GiriSrivastava2017}
D.~Giri and R.~K.~Srivastava,
\emph{Heisenberg uniqueness pairs for some algebraic curves in the plane},
Adv. Math. \textbf{310} (2017), 993--1016.
\href{https://doi.org/10.1016/j.aim.2017.02.019}
{doi:10.1016/j.aim.2017.02.019}.

\bibitem{GoncalvesRamos2022}
F.~Gon\c{c}alves and J.~P.~G.~Ramos,
\emph{A note on discrete Heisenberg uniqueness pairs for the parabola},
Bull. Sci. Math. \textbf{174} (2022), Art.~103095.
\href{https://doi.org/10.1016/j.bulsci.2021.103095}
{doi:10.1016/j.bulsci.2021.103095}.

\bibitem{GrochenigJaming2020}
K.~Gr\"ochenig and P.~Jaming,
\emph{The Cram\'er--Wold theorem on quadratic surfaces and Heisenberg
uniqueness pairs},
J. Inst. Math. Jussieu \textbf{19} (2020), no.~1, 117--135.
\href{https://doi.org/10.1017/S1474748017000457}
{doi:10.1017/S1474748017000457}.

\bibitem{Corr2026}
H.~Hedenmalm,
\emph{Corrigendum to The Klein--Gordon equation, the Hilbert transform,
and dynamics of Gauss-type maps},
J. Eur. Math. Soc. \textbf{28} (2026), no.~1, 455--457.
\href{https://doi.org/10.4171/JEMS/1689}{doi:10.4171/JEMS/1689}.

\bibitem{HM2011}
H.~Hedenmalm and A.~Montes--Rodr\'{\i}guez,
\emph{Heisenberg uniqueness pairs and the Klein--Gordon equation},
Ann. of Math. (2) \textbf{173} (2011), no.~3, 1507--1527.
\href{https://doi.org/10.4007/annals.2011.173.3.6}
{doi:10.4007/annals.2011.173.3.6}.

\bibitem{HM2020}
H.~Hedenmalm and A.~Montes--Rodr\'{\i}guez,
\emph{The Klein--Gordon equation, the Hilbert transform, and dynamics
of Gauss-type maps},
J. Eur. Math. Soc. \textbf{22} (2020), no.~6, 1703--1757.
\href{https://doi.org/10.4171/JEMS/954}{doi:10.4171/JEMS/954}.

\bibitem{HM2021}
H.~Hedenmalm and A.~Montes--Rodr\'{\i}guez,
\emph{The Klein--Gordon equation, the Hilbert transform, and Gauss-type
maps: $H^\infty$ approximation},
J. Anal. Math. \textbf{144} (2021), no.~1, 119--190.
\href{https://doi.org/10.1007/s11854-021-0173-4}
{doi:10.1007/s11854-021-0173-4}.

\bibitem{HMHyperbolic2026}
H.~Hedenmalm and A.~Montes--Rodr\'{\i}guez,
\emph{Hyperbolic Fourier series and the Klein--Gordon equation},
arXiv:2401.06871v3 (2026).
\href{https://doi.org/10.48550/arXiv.2401.06871}
{doi:10.48550/arXiv.2401.06871}.

\bibitem{Hennion1993}
H.~Hennion,
\emph{Sur un th\'eor\`eme spectral et son application aux noyaux
lipchitziens},
Proc. Amer. Math. Soc. \textbf{118} (1993), no.~2, 627--634.
\href{https://doi.org/10.1090/S0002-9939-1993-1129880-8}
{doi:10.1090/S0002-9939-1993-1129880-8}.

\bibitem{JamingKellay2018}
P.~Jaming and K.~Kellay,
\emph{A dynamical system approach to Heisenberg uniqueness pairs},
J. Anal. Math. \textbf{134} (2018), no.~1, 273--301.
\href{https://doi.org/10.1007/s11854-018-0010-6}
{doi:10.1007/s11854-018-0010-6}.

\bibitem{KNS2025}
A.~Kulikov, F.~Nazarov and M.~Sodin,
\emph{Fourier uniqueness and non-uniqueness pairs},
J. Math. Phys. Anal. Geom. \textbf{21} (2025), no.~1, 84--130.
\href{https://doi.org/10.15407/mag21.01.04}
{doi:10.15407/mag21.01.04}.

\bibitem{LasotaYorke1973}
A.~Lasota and J.~A.~Yorke,
\emph{On the existence of invariant measures for piecewise monotonic
transformations},
Trans. Amer. Math. Soc. \textbf{186} (1973), 481--488.
\href{https://doi.org/10.1090/S0002-9947-1973-0335758-1}
{doi:10.1090/S0002-9947-1973-0335758-1}.

\bibitem{RadchenkoRamos2024}
D.~Radchenko and J.~P.~G.~Ramos,
\emph{Perturbed lattice crosses and Heisenberg uniqueness pairs},
arXiv:2410.04557v1 (2024).
\href{https://doi.org/10.48550/arXiv.2410.04557}
{doi:10.48550/arXiv.2410.04557}.

\bibitem{RV2019}
D.~Radchenko and M.~Viazovska,
\emph{Fourier interpolation on the real line},
Publ. Math. Inst. Hautes \'Etudes Sci. \textbf{129} (2019), 51--81.
\href{https://doi.org/10.1007/s10240-018-0101-z}
{doi:10.1007/s10240-018-0101-z}.

\bibitem{RS2022}
J.~P.~G.~Ramos and M.~Sousa,
\emph{Fourier uniqueness pairs of powers of integers},
J. Eur. Math. Soc. \textbf{24} (2022), no.~12, 4327--4351.
\href{https://doi.org/10.4171/JEMS/1194}{doi:10.4171/JEMS/1194}.

\bibitem{Rychlik1983}
M.~Rychlik,
\emph{Bounded variation and invariant measures},
Studia Math. \textbf{76} (1983), no.~1, 69--80.
\href{https://doi.org/10.4064/sm-76-1-69-80}
{doi:10.4064/sm-76-1-69-80}.

\bibitem{Sjolin2013}
P.~Sj\"olin,
\emph{Heisenberg uniqueness pairs for the parabola},
J. Fourier Anal. Appl. \textbf{19} (2013), no.~2, 410--416.
\href{https://doi.org/10.1007/s00041-013-9258-5}
{doi:10.1007/s00041-013-9258-5}.

\end{thebibliography}
\end{document}